\documentclass[11pt]{article}

\usepackage{color}
\usepackage{graphicx}
\usepackage{amsmath}
\usepackage{cases}
\usepackage{amsthm}
\usepackage{amssymb}
\usepackage{bm}
\usepackage{cases}
\usepackage{amsfonts}
\usepackage{latexsym,enumerate,amscd,verbatim}
\usepackage{mathrsfs}
\usepackage[pdftex,colorlinks=true,bookmarks=false,citecolor=blue]{hyperref}
\usepackage{dsfont}
\usepackage{setspace}

\usepackage[figuresright]{rotating}%
\usepackage[title]{appendix}%
\usepackage{xcolor}%
\usepackage{textcomp}%
\usepackage{manyfoot}%
\usepackage{booktabs}%
\usepackage[noend]{algpseudocode}
\usepackage{algorithmicx}
\usepackage[lined]{algorithm}%
\usepackage{algorithmicx}%
\usepackage{algpseudocode}%
\usepackage{listings}%

\allowdisplaybreaks
\numberwithin{equation}{section}

\newtheorem{theorem}{Theorem}[section]
\newtheorem{lemma}{Lemma}[section]
\newtheorem{proposition}{Proposition}[section]

\newtheorem{definition}{Definition}[section]

\newtheorem{remark}{Remark}[section]

\makeatletter

\renewcommand{\fnum@algorithm}{}

\def\ns{\noalign{\ss} }
\def\ds{\displaystyle}
\def\ss{\smallskip}

\renewcommand{\@seccntformat}[1]{\csname the#1\endcsname.\hspace{0.5em}}
\makeatother

\title{Lagrangian dual method for solving stochastic linear quadratic optimal control problems with terminal state constraints}
\author{
Haisen Zhang\footnote{School of Mathematical Sciences, Sichuan Normal University, Chengdu 610066, China.
The research of this author is partially supported by NSF of China under grants 12071324 and 11931011. Email:
haisenzhang@yeah.net.}~~and~~Xianfeng Zhang\footnote{School of Mathematics, Sichuan University, Chengdu 610068, China. The research of this author is partially
supported by NSF of China under grant 11971334. Email: math\_xianfeng@163.com.}}
\date{}

\begin{document}
	\maketitle
\begin{abstract}
A stochastic linear quadratic (LQ) optimal control problem with a pointwise linear equality constraint on the terminal state is considered. A strong Lagrangian duality theorem is proved under a uniform convexity condition on the cost functional and a surjectivity condition on the linear constraint mapping. Based on the Lagrangian duality, two approaches are proposed to solve the constrained stochastic LQ problem. First, a theoretical method is given to construct the closed-form solution by the strong duality. Second, an iterative algorithm, called augmented Lagrangian  method (ALM), is proposed. The strong convergence of the iterative sequence generated by ALM is proved. In addition, some sufficient conditions for the surjectivity of the constraint mapping are obtained.
\end{abstract}


\vspace{+0.3em}


	\section{Introduction}\label{sec1}
Let $T>0$ and $(\Omega, \mathcal{F},  \mathbb{F}, \mathbb{P})$  be a complete filtered
probability space with the filtration $\mathbb{F}=\{\mathcal{F}_{t} \}_{0\le t\le T}$ (satisfying the usual conditions), on which a one-dimensional standard  Wiener
process $W(\cdot)$ is defined such that $\mathbb{F}$ is the natural filtration generated by $W(\cdot)$ (augmented by all the $\mathbb{P}$-null sets). Fix $m,n,\ell\in \mathbb{N}$. Denote by $\langle{\cdot} , {\cdot}\rangle$ and $|\cdot|$ respectively the inner product and norm in $\mathbb{R}^{m}$, $\mathbb{R}^{n}$  and $\mathbb{R}^{\ell}$, which can be identified from the contexts.

Let us consider the controlled linear stochastic differential equation
\begin{align}\label{eq controlsys}
\left\{
\begin{array}{l}	\mathrm{d} X^{x,u}(t)=\big(A(t)X^{x,u}(t)+B(t)u(t)\big)\mathrm{d}t+\big(C(t)X^{x,u}(t)+D(t)u(t)\big)\mathrm{d}W(t), \quad t\in[0,T]\\
\ns
\ds
X(0)=x
\end{array}
\right.
\end{align}
with the quadratic cost functional
\begin{equation}\label{costfun}
J(u)=\frac{1}{2}\mathbb{E}\Bigg[\int_{0}^{T}\langle Q(t)X^{x,u}(t),X^{x,u}(t)\rangle+\langle R(t)u(t),u(t)\rangle\mathrm{d}t+\langle GX^{x,u}(T),X^{x,u}(T)\rangle\Bigg]
\end{equation}
and the terminal state constraint
\begin{equation}\label{stateconstr}
M X^{x,u}(T)-b=0,\quad a.s.
\end{equation}
Here,  $A,C:[0,T]\times\Omega\to \mathbb{R}^{n\times n}$ and $B,D:[0,T]\times\Omega\to \mathbb{R}^{n\times m}$ are given matrix-valued stochastic processes, $u(\cdot)\in L^2_{\mathbb{F}}(0,T;\mathbb{R}^m)$ is the control where $L^2_{\mathbb{F}}(0,T;\mathbb{R}^m)$ is the space of $\mathbb{F}$-progressively measurable stochastic processes valued in $\mathbb{R}^m$ such that $\mathbb{E}\int_{0}^{T}|u(t)|^2dt<\infty$, $x\in \mathbb{R}^n$  and $X^{x,u}(\cdot)$ is the state valued in $\mathbb{R}^n$ with initial datum $x$ and control $u(\cdot)$. $G:\Omega\to  \mathbb{R}^{n\times n}$ is a  given matrix-valued random variable, $Q: [0,T]\times\Omega\to \mathbb{R}^{n\times n}$ and $R: [0,T]\times\Omega\to \mathbb{R}^{m\times m}$ are given matrix-valued stochastic processes. $M:\Omega\to\mathbb{R}^{\ell\times n} (\ell\leq n)$ is a given matrix-valued random variable and $b:\Omega \to\mathbb{R}^\ell$ is a given random vector.

The constrained stochastic linear quadratic (LQ) optimal control problem  considered in this paper is
\begin{equation*}
\quad \left\{
\begin{array}{ll}
	\min  &J(u),\\[+0.3em]
\text{s.t. }&u\in L^2_{\mathbb{F}}(0,T;\mathbb{R}^m),\\[+0.3em]
                  & 	M X^{x,u}(T)-b=0,\quad a.s.
\end{array}
\right. \tag{CSLQ}
\end{equation*}
The (CSLQ) is feasible if  there is a control $u(\cdot)\in L^2_{\mathbb{F}}(0,T;\mathbb{R}^m)$ with corresponding state $X^{x,u}(\cdot)$ such that the state constraint \eqref{stateconstr} is satisfied. Any minimizer $\bar{u}(\cdot)$ of (CSLQ) is called an optimal control, the corresponding state process $\bar{X}^{x,\bar{u}}(\cdot)$ is called an optimal state, and $(\bar{X}^{x,\bar{u}}(\cdot), \bar{u}(\cdot))$ is called an optimal pair.

Stochastic LQ problem is one of the fundamental problems in stochastic control theory and has wide range of applications in many fields, such as engineering, management science and mathematical finance. The stochastic  LQ problem without constraint was initiated by Wonham \cite{Wonham1968} and  studied extensively for both deterministic and random coefficients by many researchers in the past few decades. For instance, Bismut \cite{Bismut1976} first studied the stochastic LQ problem with random coefficients. Chen, Li and Zhou \cite{ChenLiZhou1998} found for the first time that stochastic LQ problem with indefinite control weight cost may still be well-posed, which is essentially different from its deterministic counterpart. Rami, Moore and Zhou \cite{RamiMooreZhou2001} proposed a much general Riccati equation with an additional algebraic equality constraint and proved that the solvability of the generalized Riccati equation is sufficient and necessary for the well-posedness of the indefinite stochastic  LQ problem with deterministic coefficients. Tang \cite{Tang2003,Tang2015} proved the existence and uniqueness of the  solution to the backward stochastic Riccati equation for stochastic LQ problem with random coefficients under the regular case that the wight matrix $R$ is uniformly positive definite.   Kohlmann and Tang \cite{KohlmannTang2003}, Hu and Zhou \cite{HuZhou2003} studied the the existence and uniqueness of the  solution to the backward stochastic Riccati equation for stochastic LQ problem with random coefficients in some special indefinite cases. Sun, Li and Yong  \cite{SunLiYong2016}, Sun, Xiong and Yong \cite{SunXiongYong2021}, Sun and Yong \cite{SunYong2020book} studied the relationship between the open-loop solvability and closed-loop solvability for stochastic LQ problems.

In many applications, the control or/and  the state of the control system should satisfy some constraints. Obviously, solving the constrained stochastic LQ problems is more challenging than solving the problems without constraints.  The stochastic linear quadratic optimal problem with cone control constraints and random coefficients was studied by  Hu and Zhou \cite{HuZhou2005}.  An explicit optimal feedback control was obtained in \cite{HuZhou2005} by introducing two extended stochastic Riccati equations. Chen and Zhou \cite{ChenZhou2004} discussed the  stochastic linear quadratic optimal problem in infinite time horizon with conic control constraints.  Recently, Hu, Shi and Xu studied  in \cite{HuShiXu2022aap}  and  \cite{HuShiXu2022cocv} respectively the finite time horizon and infinite time horizon stochastic linear-quadratic optimal control problem with regime switching, random coefficients and cone control constraints. Lim and Zhou \cite{LimZhou1999} studied the stochastic LQ problem with mixed control-state integral type  quadratic inequality constraints.  Wu, Gao, Lu and Li \cite{WuGaoLuLi2020} discussed the scalar-state stochastic LQ optimal control problems with mixed pointwise state-control linear inequality constraints. Feng, Hu and Huang \cite{FengHuHuang2022} considered a stochastic LQ problem with a terminal state affine expectation constraint  when they studied the backward Stackelberg differential game involving a single leader and single follower.

In contrast with the stochastic LQ problems with inequality   state constraints or  mixed control-state  inequality   constraints, less progress has been made on the stochastic LQ problems with equality  state constraints. Lim \cite{Lim2004} gave a closed-form expression of the optimal portfolios for mean-variance portfolio selection problems in which the state $X^{x,u}(\cdot)$ is a real-scalar process and constrained by the expectation type equality constraint
\begin{equation}\label{exp_cons}
\mathbb{E}X^{x,u}(T)=c
\end{equation}
for some constant $c\in \mathbb{R}$.
Kohlmann and Tang \cite[Section 6.2]{KohlmannTang2003}  studied the general multi-dimensional stochastic LQ problem with state constraint \eqref{exp_cons} (in which $X^{x,u}(T)$ is valued in $\mathbb{R}^{n}$ and $c\in \mathbb{R}^{n}$). A feedback solution with parameter for the optimal control was given by the Riccati equation. Zhang and Zhang \cite{zhang} studied the stochastic LQ problem with state constraint
\begin{equation}\label{Estateconstr}
\mathbb{E}(M X^{x,u}(T)-b)=0
\end{equation}
under the solvability
condition on a stochastic Riccati equation and a surjectivity condition on the linear constraint mapping.  The equivalence between the surjectivity condition
and a Kalman- type rank condition is proved in \cite{zhang} for the special case of deterministic coefficients. In both \cite{Lim2004} and \cite{zhang}, the Lagrangian duality is the main tool to handle the state constraint and the optimal parameter of the feedback optimal solution is determined explicitly by solving the dual problem.

Compared with the expectation type terminal state constraint \eqref{Estateconstr}, the stochastic LQ problems with (sample point) pointwise terminal state constraints are more complicated. First, the pointwise type state constraints are more restrictive than the expectation type state constraints and hence some stronger conditions are needed to ensure the feasibility of the correlated state constrained problems. Second, the dual variable for the  pointwise type  state constraint \eqref{stateconstr} is a random vector and the corresponding dual problem is a stochastic programming problem which is hard to be solved explicitly in general.

In \cite{BiSunXiong2020}, Bi, Sun and Xiong  used the BSDE approach to study a  stochastic linear quadratic optimal problem with deterministic coefficients, fixed terminal state and a pointwise linear equality constraint on the initial state. A closed-form solution is obtained by solving a Riccati equation and an algebraic matrix equation for the optimal parameter. Besides, there are a few articles on stochastic LQ problems with pointwise equality constraints for the special case of the norm optimal control problem, i.e., $M\equiv I_{n}$, $Q(t)\equiv 0$, $R(t)\equiv I_{m}$, $G\equiv 0$ ($I_{n}$ and $I_{m}$ are respectively the identity matrices of size $n$ and $m$). See for instance, Gashi \cite{Gashi2015}, Wang and Zhang \cite{WangZhang2015}, Wang, Yang, Yong and Yu \cite{WangYangYongYu2017}.

In this paper, we shall discuss the stochastic LQ  problem with pointwise linear equality constraint \eqref{stateconstr}. Similar to our previous work \cite{zhang} for expectation type terminal state constraint \eqref{Estateconstr}, we prove a strong Lagrangian duality theorem for the constrained stochastic LQ  problem (CSLQ) under a uniform convexity condition on the cost functional and a surjectivity condition on the linear constraint mapping. An equivalent characterization of the surjectivity condition is discussed by the controllability theory of linear control systems. In addition, a Kalman-type rank condition, which is sufficient for the surjectivity condition, is derived in the special case of deterministic  coefficients. Different from \cite{zhang},  the dual problem for (CSLQ) cannot be solved explicitly by its first-order necessary condition. As a result, the closed-form solution to  (CSLQ) cannot be obtained directly by the dual approach. See Section \ref{sec3} for more details. To overcome that difficulty, we introduce an iterative algorithm, called augmented Lagrangian method (ALM), to solve the constrained stochastic LQ  problem (CSLQ).

The ALM  is originally proposed independently by Hestenes \cite{Hestentes1969} and Powell \cite{Powell1972}  for solving finite dimensional constrained optimization problems and has been extensively studied by many scholars in the past few decades. We refer the reader to \cite{Kunisch1997,Glowinski1983,ItoKunisch1990} and the references cited therein  for its infinite dimensional extensions and applications in numerical solution to partial differential equations and deterministic optimal control problems. Recently, Pfeiffer \cite{Pfeiffer2020} proposed an ALM for solving non-linear stochastic control problems with inequality type terminal state constraints. The cost functional and the inequality constraints in \cite{Pfeiffer2020} are functional of the probability distribution of the terminal state.

In this paper, we show that the ALM is effective to solve the constrained stochastic LQ  problem (CSLQ). Under proper conditions, we prove that the iterative sequence generated by ALM  converges strongly to the optimal control of (CSLQ). The basic idea is form the ALM for the quadratic programming problem with linear equality constraints(see, for instance, Chapter 1 in \cite{Glowinski1983}). Indeed, the cost functional \eqref{costfun} can be represented as a quadratic functional of control though introducing some proper operators, for more details we refer the readers to \cite[Theorem 3.4]{SunXiongYong2021}. Then, the convergence of the ALM for (CSLQ) can be obtained by verifying the convergence conditions of the ALM for quadratic programming problem with linear equality constraints. Instead of using such an abstract approach, in this paper we prove the convergence result directly by the elementary techniques in stochastic control.

The main contributions of this paper are as follows:

\begin{enumerate}[(i)]
  \item The  Lagrangian dual method is proposed to solve the constrained stochastic LQ problem  (CSLQ). First, the closed-form solution with optimal parameter is constructed by the Lagrangian duality principle.
      Then, the ALM is introduced to solve (CSLQ) and the strong convergence of the iterative sequence is proved in a simple and direct way.
  \item Some verifiable sufficient conditions are given to ensure the strong duality between the (CSLQ) and its dual problem. Those sufficient conditions are also the convergence conditions of the ALM for (CSLQ).
  \item As a byproduct of  the convergence proof of the ALM,  a first-order necessary and sufficient condition for the optimal control of (CSLQ) is obtained by the Lagrangian duality theory (See Lemma \ref{lemmasss}).
\end{enumerate}

The rest of this paper is organized as follows. In Section 2 we introduce some basic notations and assumptions. In Section 3 we prove the Lagrangian duality between the (CSLQ) and its dual problem under a uniform convexity condition on the cost functional and a surjectivity condition on the linear
constraint mapping. In Section 4, we   propose the ALM  for (CSLQ)  and prove its strong convergence. Finally  we give some verifiable sufficient conditions for the surjectivity condition of the linear constraint mapping in Section 5.

\section{Preliminaries and assumptions}\label{sec2}
	
Throughout this paper, let $\mathbb{R}^n$, $\mathbb{R}^m$ and $\mathbb{R}^\ell$  be respectively the $n$, $m$ and $\ell$-dimensional Euclidean space. Let $\mathbb{R}^{n\times n}$ and $\mathbb{R}^{n\times m}$ be respectively the sets of all  $n\times n$ and $n\times m$ real matrices. Denoted by $M^{\top}$ the transpose of $M$, by $\mathbb{S}^n$ the space of all symmetric $n\times n$ real matrices. The identity matrix of size $n$ is denoted by $I_{n}$. For $M,N\in \mathbb{S}^n$, denote $M\geq N$ when $M-N$ is positive semidefinite.

For a Banach space $\mathds{X}$ with its norm $|\cdot|_{\mathds{X}}$, denote by $\mathcal{B}_{\mathds{X}}(0,1)$ the open unit ball of $\mathds{X}$. Denote by $L^2_{\mathcal{F}_{T}}(\Omega;\mathds{X})$   the space of $\mathds{X}$-valued, $\mathcal{F}_{T}$ measurable random vectors $\xi$ such that $\|\xi\|_{L_{\mathcal{F}_{T}}^2(\Omega;\mathds{X})}\!\triangleq\!\big[\mathbb{E}|\xi|_{\mathds{X}}^2\big]^{\frac{1}{2}}\!<\!\infty$; by $L^\infty_{\mathcal{F}_{T}}(\Omega;\mathds{X})$ the space of $\mathds{X}$-valued, $\mathcal{F}_{T}$ measurable random vectors $\xi$ such that $\mathop{esssup}_{\omega}|\xi|_{\mathds{X}} \!<\!\infty$; by $L^2_{\mathbb{F}}(0,T;\mathds{X})$ the space of $\mathds{X}$-valued, $\mathbb{F}$-progressively measurable stochastic processes $\eta$ such that $\| \eta\|_{L^{2}_{\mathbb{F}}(0,T;\mathds{X})}\triangleq\big[\mathbb{E}\int_{0}^{T}|\eta(t)|_{\mathds{X}}^2dt\big]^{\frac{1}{2}}<\infty$; by $L^\infty_{\mathbb{F}}(0,T;\mathds{X})$ the space of $\mathds{X}$-valued, $\mathbb{F}$-progressively measurable stochastic processes $\eta$ such that $\mathop{esssup}_{(t,\omega)} |\eta(t)|_{\mathds{X}}<\infty$; by $L_{\mathbb{F}}^2\big(\Omega; C\big([0,T];\mathds{X}\big)\big)$ the space of $\mathds{X}$-valued, $\mathbb{F}$-progressively measurable continuous stochastic processes $\eta$ such that $\big[\mathbb{E}\big(\sup_{0\leq t\leq T}|\eta(t)|_{\mathds{X}}^2\big)\big]^{\frac{1}{2}}<\infty$;
by $L_{\mathbb{F}}^{\infty}\left(\Omega ; C\left([0, T], \mathds{X}\right)\right)$ the space of $\mathds{X}$-valued, $\mathbb{F}$-progressively measurable continuous stochastic processes $\eta$ such that $\mathop{esssup}_{\omega}\big(\sup_{0\leq t\leq T}|\eta(t)|_{\mathds{X}}\big)<\infty$;
by $L_{\mathbb{F}}^{\infty}\left(\Omega; L^{2}\left(0, T ; \mathds{X}\right)\right)$ the space of $\mathds{X}$-valued, $\mathbb{F}$-progressively measurable stochastic processes $\eta$ such that $\mathop{esssup}_{\omega}\big(\int_{0}^{T}|\eta(t)|_{\mathds{X}}^2dt\big)^{\frac{1}{2}}<\infty$.\\

Throughout this paper, we make the following assumptions.
\begin{enumerate}
  \item [ (A1)] $
		A(\cdot),C(\cdot)\in L^{\infty}_{\mathbb{F}}(0,T;\mathbb{R}^{n\times n})$, \qquad $B(\cdot), D(\cdot)\in L^{\infty}_{\mathbb{F}}(0,T;\mathbb{R}^{n\times m})$.
  \item [ (A2)] $G\in L^{\infty}_{\mathcal{F}_{T}}(\Omega;\mathbb{S}^n)$,\quad $Q(\cdot)\in L^{\infty}_{\mathbb{F}}(0,T;\mathbb{S}^n)$,\quad
		$R(\cdot)\in L^{\infty}_{\mathbb{F}}(0,T; \mathbb{S}^m)$, \quad $M\in L^{\infty}_{\mathcal{F}_{T}}(\Omega;\mathbb{R}^{\ell \times n})$.

  \item [(A3)] There is a constant $\delta>0$ such that
\begin{eqnarray}\label{J0}
J^{0}(v)\!\!\!&\triangleq&\!\!\!\frac{1}{2}\mathbb{E}\Bigg[\int_{0}^{T}\langle Q(t)X^{0,v}(t),X^{0,v}(t)\rangle+\langle R(t)v(t),v(t)\rangle\mathrm{d}t+\langle GX^{0,v}(T),X^{0,v}(T)\rangle\Bigg]\nonumber \\[+0.5em]
\!\!\!&\ge&\!\!\! \delta\mathbb{E}\int_{0}^{T}|v(t)|^2dt,\quad \forall v\in  L^2_{\mathbb{F}}(0,T;\mathbb{R}^m).
\end{eqnarray}
Here, $X^{0,v}(\cdot)$ is the solution to control system \eqref{eq controlsys} with control $v\in  L^2_{\mathbb{F}}(0,T;\mathbb{R}^m)$ and initial datum $0$.

  \item [ (A4)] For the given matrix-valued random variable $M\in L^{\infty}_{\mathcal{F}_{T}}(\Omega;\mathbb{R}^{\ell \times n})$ and initial datum $x\in \mathbb{R}^n$, the mapping $u\mapsto MX^{x,u}(T)$ is surjective, i.e.,
	\begin{align*}
		L^{2}_{\mathcal{F}_{T}}(\Omega;\mathbb{R}^{\ell})=\Big\{MX^{x,u}(T)\ \Big|\ u\in L^2_{\mathbb{F}}(0,T;\mathbb{R}^m)\Big\}.
	\end{align*}
\end{enumerate}

By condition (A4), we have the  set of
admissible controls
\begin{equation}\label{uad}
U_{ad}\triangleq\Big\{u\in L^2_{\mathbb{F}}(0,T;\mathbb{R}^m)\ \Big|\ MX^{x,u}(T)-b=0,\ a.s. \Big\}
\end{equation}
is nonempty. Then, by (A1), (A2) and (A4),
the constrained stochastic LQ problem (CSLQ) is well-defined, i.e., for any $u\in U_{ad}$, state equation \eqref{eq controlsys} admits a unique solution $X^{x,u}$ and $J(u)< +\infty$. In addition,  we shall see that the condition (A3), which is called uniform convexity condition in \cite{SunXiongYong2021}, implies the strong convexity of the cost functional $J(\cdot)$. Then, under conditions (A1)--(A4), the constrained stochastic LQ problem (CSLQ) admits unique optimal solution.

\begin{definition}\label{2222}
Let $\mathds{X}$ be a Banach space, $f:\mathds{X}\to \mathbb{R}$ is called a strongly convex functional with constant $\sigma>0$  if
$$f\big(\theta x+(1-\theta)y\big)\le \theta f(x)+(1-\theta)f(y)-\frac{\sigma}{2}\theta(1-\theta)|x-y|^2_{\mathds{X}},\quad \forall\ x,y\in \mathds{X}, \theta\in[0,1].$$
\end{definition}

\begin{lemma}\label{lemma existence}
Suppose that   (A1)--(A4) hold. Then the cost functional $J(\cdot)$ is a strongly convex continuous  functional on $L^2_{\mathbb{F}}(0,T;\mathbb{R}^m)$ and the constrained stochastic LQ problem (CSLQ) is uniquely solvable.
\end{lemma}

\begin{proof}
The continuity of $J(\cdot)$ is obvious.
For any $u_{1}, u_{2}\in L^2_{\mathbb{F}}(0,T;\mathbb{R}^m)$, $\theta\in [0,1]$, by (A1)--(A3), we have
\begin{align}\label{eq stronglyconvex}
J\big(\theta u_{1}+(1-\theta)u_{2}\big)&= \frac{1}{2} \mathbb{E}\Bigg[\int_{0}^{T}\left\langle Q(t) X^{x,\theta u_{1}+(1-\theta)u_{2}}(t), X^{x,\theta u_{1}+(1-\theta)u_{2}}(t)\right\rangle\nonumber\\
&\qquad\qquad +\langle R(t) (\theta u_{1}(t)+(1-\theta)u_{2}(t)), \theta u_{1}(t)+(1-\theta)u_{2}(t)\rangle \mathrm{d} t\nonumber\\
&\qquad\qquad  +\left\langle G X^{x,\theta u_{1}+(1-\theta)u_{2}}(T), X^{x,\theta u_{1}+(1-\theta)u_{2}}(T)\right\rangle\Bigg] \nonumber\\
&= \theta J(u_{1})+(1-\theta)J(u_{2}) \nonumber\\ &\quad -\frac{1}{2}\theta(1-\theta)\mathbb{E}\Bigg[\int_{0}^{T}\left\langle Q(t) X^{0, u_{1}-u_{2}}(t), X^{0, u_{1}-u_{2}}(t)\right\rangle\nonumber\\
&\quad +\langle R(t) (u_{1}(t)-u_{2}(t)),  u_{1}(t)-u_{2}(t)\rangle \mathrm{d} t\nonumber\\
&\quad  +\left\langle G X^{0, u_{1}-u_{2}}(T), X^{0, u_{1}-u_{2}}(T)\right\rangle\Bigg] \nonumber\\
&=\theta J(u_{1})+(1-\theta)J(u_{2})-\frac{1}{2}\theta(1-\theta)J^0(u_{1}-u_{2})\nonumber\\
&\le \theta J(u_{1})+(1-\theta)J(u_{2})-
\frac{\delta}{2}\theta(1-\theta)\mathbb{E}\int_{0}^{T}|u_{1}(t)-u_{2}(t)|^2dt,
\end{align}
i.e., $J(\cdot)$ is a strongly convex  functional on $L^2_{\mathbb{F}}(0,T;\mathbb{R}^m)$.

By assumption (A4), $U_{ad}$  is nonempty. Since the control system \eqref{eq controlsys} is linear and the terminal state constraint is a linear equality constraint, $U_{ad}$ is a closed convex subset of $L^2_{\mathbb{F}}(0,T;\mathbb{R}^m)$. Then, by the standard existence theory of convex optimization (see, for instance, \cite[Theorem 2.31]{BonnansShapiro2000}), the problem (CSLQ) is uniquely solvable.
\end{proof}

\section{Lagrangian duality}\label{sec3}

In this section, we shall prove a Lagrangian duality theorem for the constrained stochastic LQ problem (CSLQ) and derive a closed-form solution with optimal parameter to (CSLQ) by dual approach.

Let us first recall some basic notions for the Lagrangian duality in optimization. For more details  we refer the readers to \cite{BonnansShapiro2000}. Let $\mathds{X}$, $\mathds{Y}$ be two Banach spaces, $C_{\mathds{X}}\subset \mathds{X}$, $C_{\mathds{Y}}\subset \mathds{Y}$ be arbitrary nonempty sets. Let us associate with a functional $\mathds{L}: C_{\mathds{X}}\times C_{\mathds{Y}} \to \mathbb{R}\cup\{\pm\infty\}$ the primal and dual problems, defined as follows
\begin{align*}
		\inf_{x\in C_{\mathds{X}}}\sup_{y\in C_{\mathds{Y}}  }\mathds{L}(x,y), \tag{P}
	\end{align*}
\begin{align*}
		\sup_{y\in C_{\mathds{Y}} }\inf_{x\in C_{\mathds{X}}}\mathds{L}(x,y).\tag{D}
\end{align*}

\begin{definition}[\cite{BonnansShapiro2000}]\label{11111}
It is said that the strong duality holds between the problem (P) and problem (D) if both problems have finite optimal values and
$$\sup_{y\in C_{\mathds{Y}} }\inf_{x\in C_{\mathds{X}}}\mathds{L}(x,y)= \inf_{x\in C_{\mathds{X}}}\sup_{y\in C_{\mathds{Y}}  }\mathds{L}(x,y).$$
$(\bar{x}, \bar{y})\in C_{\mathds{X}} \times C_{\mathds{Y}}$ is called a saddle point of the functional $\mathds{L}$ if $\mathds{L}(\bar{x}, \bar{y})\in \mathbb{R} $ and
	\begin{align*}
		\mathds{L}(\bar{x}, y)\leq \mathds{L}(\bar{x}, \bar y)\leq \mathds{L}(x, \bar y),\quad \forall\ (x,y)\in C_{\mathds{X}} \times C_{\mathds{Y}}.
	\end{align*}
\end{definition}

Now let us consider the Lagrangian duality theory for the constrained stochastic LQ problem (CSLQ). Define the Lagrangian functional for (CSLQ) by
\begin{equation*}
L(u,\lambda)\triangleq J(u)+\mathbb{E}\langle \lambda,MX^{x,u}(T)-b\rangle, \quad \forall \ u\in L^{2}_{\mathbb{F}}(0,T;\mathbb{R}^m), \ \lambda\in L^{2}_{\mathcal{F}_{T}}(\Omega;\mathbb{R}^{\ell}).
\end{equation*}
Here, $J(\cdot)$ is the cost functional defined by \eqref{costfun}. Clearly,
\begin{align*}
\sup_{\lambda \in L^{2}_{\mathcal{F}_{T}}(\Omega;\mathbb{R}^{\ell})}L(u,\lambda)&=\sup_{\lambda \in L^{2}_{\mathcal{F}_{T}}(\Omega;\mathbb{R}^{\ell})}\Big\{J(u)+\mathbb{E}\langle \lambda,MX^{x,u}(T)-b\rangle\Big\}\\
&=\begin{cases}
J(u),&   MX^{x,u}(T)-b=0,\ a.s. \\
+\infty,&  MX^{x,u}(T)-b\neq0,\ a.s.
\end{cases}
\end{align*}
Thus, the problem (CSLQ) is equivalent to
\begin{align}\label{pirmal_for_SLQ}
\inf_{u\in U_{ad}}J(u)=\inf_{\substack{ u\in L^{2}_{\mathbb{F}}(0,T;\mathbb{R}^m)\\ MX^{x,u}(T)-b=0}}J(u)=\inf_{u\in L^{2}_{\mathbb{F}}(0,T;\mathbb{R}^m)}\sup_{\lambda \in L^{2}_{\mathcal{F}_{T}}(\Omega;\mathbb{R}^{\ell})}L(u,\lambda).
\end{align}
Define the dual functional $d:L^{2}_{\mathcal{F}_{T}}(\Omega;\mathbb{R}^{\ell})\to \mathbb{R}\cup \{\pm\infty\}$ by
\begin{equation}\label{dual_func}
d(\lambda)\triangleq\inf_{u\in L^{2}_{\mathbb{F}}(0,T;\mathbb{R}^m)}L(u,\lambda), \quad \forall\ \lambda\in L^{2}_{\mathcal{F}_{T}}(\Omega;\mathbb{R}^{\ell}),
\end{equation}
and define the dual problem for (CSLQ) by
\begin{equation}\label{dual_for_SLQ}
\sup_{\lambda \in L^{2}_{\mathcal{F}_{T}}(\Omega;\mathbb{R}^{\ell})}\inf_{u\in L^{2}_{\mathbb{F}}(0,T;\mathbb{R}^m)}L(u,\lambda)=\sup_{\lambda \in L^{2}_{\mathcal{F}_{T}}(\Omega;\mathbb{R}^{\ell})}d(\lambda).
\end{equation}

Since the cost functional $J(\cdot)$ is strongly convex under conditions (A1)--(A3),  $L(\cdot,\lambda)$ is also a strongly convex functional for any $\lambda\in L^{2}_{\mathcal{F}_{T}}(\Omega;\mathbb{R}^{\ell})$. Then, the unconstrained stochastic LQ problem in the definition of $d(\lambda)$ admits unique solution and the dual functional $d(\cdot)$ is well-defined.  In what follows, we prove the strong duality between (CSLQ) and its dual problem \eqref{dual_for_SLQ}.

\begin{theorem}\label{th strong dual}
Suppose that (A1)--(A4) hold true and let $\bar{u}$ be the unique solution to (CSLQ).  Then the following two assertions hold true.
\begin{enumerate}[ (i) ]
\item   The strong duality between (CSLQ) and its dual problem (\ref{dual_for_SLQ}) holds true, i.e.
\begin{equation*}
\sup_{\lambda\in L^{2}_{\mathcal{F}_{T}}(\Omega;\mathbb{R}^{\ell})}\inf_{u\in L^2_{\mathbb{F}}(0,T;\mathbb{R}^m)}L(u,\lambda)= \inf_{u\in L^2_{\mathbb{F}}(0,T;\mathbb{R}^m)} \sup_{\lambda\in L^{2}_{\mathcal{F}_{T}}(\Omega;\mathbb{R}^{\ell})}L(u,\lambda).
\end{equation*}
\item The dual problem is solvable, and, if $\bar{\lambda}$ is the solution to the dual problem  then $(\bar{u},\bar{\lambda})$ is a saddle point of $L$, i.e.
\begin{equation*}
L(\bar{u},\lambda)\leq L(\bar{u},\bar{\lambda})\leq L(u,\bar{\lambda}),\quad \forall\ u\in L^2_{\mathbb{F}}(0,T;\mathbb{R}^m),   \forall\ \lambda\in L^{2}_{\mathcal{F}_{T}}(\Omega;\mathbb{R}^{\ell}).
\end{equation*}
Especially,
\begin{equation}\label{USLQ}
L(\bar u,\bar{\lambda})=\inf\limits_{u\in L^2_{\mathbb{F}}(0,T;\mathbb{R}^m)}L(u,\bar{\lambda}).
\end{equation}
\end{enumerate}
\end{theorem}
\begin{proof}
Define
$$
\mathcal{K}=\Big\{(\alpha,\beta)\in\mathbb{R}\times L^{2}_{\mathcal{F}_{T}}(\Omega;\mathbb{R}^{\ell}) \Big|\ \exists~u\in L^2_{\mathbb{F}}(0,T;\mathbb{R}^m) \text{ s.t. }J(u)-J(\overline{u})\leq \alpha, \ MX^{x,u}(T)-b=\beta,\ a.s. \Big\},
$$
and
$$
\mathcal{O}=\Big\{(\alpha^{'},\beta^{'})\in\mathbb{R}\times L^{2}_{\mathcal{F}_{T}}(\Omega;\mathbb{R}^{\ell})\Big|\alpha^{'}<0, \ \beta^{'}=0,\ a.s. \Big\}.
$$
Clearly, both $\mathcal{K}$ and $\mathcal{O}$ are convex sets. We claim that the interior of $\mathcal{K}$ is nonempty. By condition (A4), $u\mapsto MX^{x,u}(T)$ is a surjection. Then, the linear mapping $\Gamma:u\mapsto MX^{0,u}(T)$ is also a surjection. Meanwhile, there exist $\kappa_{1}, \kappa_{2}>0$ satisfying
\begin{align*}
\mathbb{E}|MX^{0,u}(T)|^2&\le\mathbb{E}|M|^2\cdot|X^{0,u}(T)|^2\le\kappa_{1} \mathbb{E}|X^{0,u}(T)|^2\le\kappa_{2} \mathbb{E}\int_{0}^{T}|u(t)|^2\mathrm{d}t.
\end{align*}
According to the classical open mapping theorem (see, for instance, \cite[Theorem 5A.1]{Dontchev2014}), we know that $\Gamma$ is an open mapping and there is~$\kappa_3>0$  such that for any~$\xi\in L_{\mathcal{F}_{T}}^2(\Omega;\mathbb{R}^\ell)$  there exists~$u\in L^2_{\mathbb{F}}(0,T;\mathbb{R}^m)$  satisfying
$MX^{0,u}(T)=\xi$, and $ \|u\|_{L^2_{\mathbb{F}}(0,T;\mathbb{R}^m)}\le\kappa_3\|\xi\|_{L_{\mathcal{F}_{T}}^2(\Omega;\mathbb{R}^\ell)}.
$
Especially, for any fixed $\varepsilon$ and any $\beta\in \varepsilon\mathcal{B}_{L_{\mathcal{F}_{T}}^2(\Omega;\mathbb{R}^\ell)}(0,1)$, there exists~$v\in L^2_{\mathbb{F}}(0,T;\mathbb{R}^m)$ such that
\begin{align*}
MX^{0,v}(T)=\beta \text{ and } \|v\|_{L^2_{\mathbb{F}}(0,T;\mathbb{R}^m)}\le\kappa_3\varepsilon.
\end{align*}
Then,
\begin{equation}\label{eq 3.8}
MX^{x,\bar{u}+v}(T)-b=M\bar{X}^{x,\bar{u}}(T)-b+MX^{0,v}(T)=\beta\in \varepsilon\mathcal{B}_{L_{\mathcal{F}_{T}}^2(\Omega;\mathbb{R}^\ell)}(0,1).
\end{equation}
Let $\alpha>0$. By the continuity of $J(\cdot)$, there is $\varepsilon$ such that
\begin{equation}\label{eq 3.9}
J(\bar{u}+v)\le J(\bar{u})+\alpha,\quad\forall\ v\in \varepsilon\mathcal{B}_{L_{\mathcal{F}_{T}}^2(\Omega;\mathbb{R}^\ell)}(0,1).
\end{equation}
Combining \eqref{eq 3.8} with \eqref{eq 3.9}, we obtain that
$$
(\alpha,+\infty)\times \varepsilon\mathcal{B}_{L_{\mathcal{F}_{T}}^2(\Omega;\mathbb{R}^\ell)}(0,1)\subset \mathcal{K}.
$$
This proves that the interior of  $\mathcal{K}$ is nonempty.

By the optimality of~$\bar{u}$, we obtain~$\mathcal{K}\cap \mathcal{O}=\emptyset$. Then, by separation theorem, there is~$(\lambda_{0},\lambda)\in \mathbb{R}\times L_{\mathcal{F}_{T}}^2(\Omega;\mathbb{R}^\ell), (\lambda_{0},\lambda)\not=0$ such that
\begin{align*}
\inf_{(\alpha,\beta)\in \mathcal{K}}\Big\{\lambda_{0}\alpha+\mathbb{E}\langle\lambda,\beta\rangle\Big\}\geq\sup_{(\alpha^{'},\beta^{'})\in \mathcal{O}}\Big\{\lambda_{0}\alpha^{'}+\mathbb{E}\langle\lambda,\beta^{'}\rangle\Big\}=\sup_{\alpha^{'}<0}\lambda_{0}\alpha^{'}.
\end{align*}
Clearly, $\lambda_{0}\geq 0, \sup\limits_{\alpha^{'}<0}\lambda_{0}\alpha^{'}=0$, and
\begin{align*}
0 \leq \inf_{(\alpha,\beta)\in \mathcal{K}}\Big\{\lambda_{0}\alpha+\mathbb{E}\langle\lambda,\beta\rangle\Big\}\leq \inf_{u\in L^2_{\mathbb{F}}(0,T;\mathbb{R}^m)}\Big\{\lambda_{0}\big(J(u)-J(\bar{u})\big)+\mathbb{E}\langle\lambda, MX^{x,u}(T)-b\rangle\Big\}.
\end{align*}

Assume $\lambda_{0}=0$, then
\begin{align*}
0\leq \mathbb{E}\langle\lambda, MX^{x,u}(T)-b\rangle,\quad \forall \ u\in L^2_{\mathbb{F}}(0,T;\mathbb{R}^m).
\end{align*}
By condition  (A4), we must have $\lambda=0$  which  contradicts to $(\lambda_{0},\lambda)\not=0$. Therefore, $\lambda_{0}>0$.
Let $\bar{\lambda}=\frac{\lambda}{\lambda_{0}}$, we obtain that
\begin{align*}
0\leq J(u)-J(\bar{u})+\mathbb{E}\langle \bar{\lambda}, MX^{x,u}(T) -b\rangle,\quad \forall\ u\in L^2_{\mathbb{F}}(0,T;\mathbb{R}^m).
\end{align*}
Since $M\bar{X}^{x,\bar{u}}(T)-b=0$,
then
\begin{equation}\label{eq 20}
J(\bar{u})+ \mathbb{E}\langle \bar{\lambda}, M\bar{X}^{x,\bar{u}}(T)-b\rangle \leq J(u)+\mathbb{E}\langle \bar{\lambda}, MX^{x,u}(T)-b\rangle, \quad \forall\ u\in L^2_{\mathbb{F}}(0,T;\mathbb{R}^m).
\end{equation}
By \eqref{pirmal_for_SLQ}--\eqref{dual_for_SLQ} and \eqref{eq 20}, we have
\begin{align}\label{eq 21}
\inf_{u\in L^2_{\mathbb{F}}(0,T;\mathbb{R}^m)}\sup_{\lambda\in L_{\mathcal{F}_{T}}^2(\Omega;\mathbb{R}^\ell)}L(u,\lambda)
=&J(\bar{u})\nonumber\\
=&J(\bar{u})+ \mathbb{E}\langle \bar{\lambda}, M\bar{X}^{x,\bar{u}}(T)-b\rangle\nonumber\\
\leq&\inf_{u\in L^2_{\mathbb{F}}(0,T;\mathbb{R}^m)}L(u,\bar{\lambda})\nonumber\\
=&d(\bar{\lambda})\nonumber\\
\leq&  \sup_{\lambda\in L_{\mathcal{F}_{T}}^2(\Omega;\mathbb{R}^\ell)}d(\lambda)\nonumber\\
=&\sup_{\lambda\in L_{\mathcal{F}_{T}}^2(\Omega;\mathbb{R}^\ell)}\inf_{u\in L^2_{\mathbb{F}}(0,T;\mathbb{R}^m)}L(u,\lambda).
\end{align}
On the other hand, it is obvious that
\begin{equation*}
\sup_{\lambda\in L_{\mathcal{F}_{T}}^2(\Omega;\mathbb{R}^\ell)}\inf_{u\in L^2_{\mathbb{F}}(0,T;\mathbb{R}^m)}L(u,\lambda)
 \leq \inf_{u\in L^2_{\mathbb{F}}(0,T;\mathbb{R}^m)}\sup_{\lambda\in L_{\mathcal{F}_{T}}^2(\Omega;\mathbb{R}^\ell)}L(u,\lambda).
\end{equation*}
Therefore,
\begin{equation}\label{eq 23}
\sup_{\lambda\in L_{\mathcal{F}_{T}}^2(\Omega;\mathbb{R}^\ell)}\inf_{u\in L^2_{\mathbb{F}}(0,T;\mathbb{R}^m)}L(u,\lambda)= \inf_{u\in L^2_{\mathbb{F}}(0,T;\mathbb{R}^m)} \sup_{\lambda\in L_{\mathcal{F}_{T}}^2(\Omega;\mathbb{R}^\ell)}L(u,\lambda).
\end{equation}
This proves (i).

In addition, by \eqref{eq 21} and   \eqref{eq 23},
\begin{equation}\label{24}
d(\bar{\lambda})=\sup_{\lambda\in L_{\mathcal{F}_{T}}^2(\Omega;\mathbb{R}^\ell)}d(\lambda),
\end{equation}
i.e., the dual problem  is solvable and $\bar \lambda$  is an optimal solution to dual problem. Furthermore, for any solution $\bar{\lambda}$ of the dual problem,
\begin{equation}\label{eq 3.11+}
 \sup_{\lambda\in L_{\mathcal{F}_{T}}^2(\Omega;\mathbb{R}^\ell)}\inf_{u\in L^2_{\mathbb{F}}(0,T;\mathbb{R}^m)} L(u,\lambda)=\inf_{u\in L^2_{\mathbb{F}}(0,T;\mathbb{R}^m)}L(u,\bar\lambda)\le L(u,\bar\lambda),\quad \forall\ u\in L^2_{\mathbb{F}}(0,T;\mathbb{R}^m).
\end{equation}
By \eqref{pirmal_for_SLQ} and the optimality of $\bar u$,
\begin{equation}\label{eq 3.12+}
L(\bar{u},\lambda)\leq L(\bar{u},\bar\lambda)=J(\bar u)= \inf_{u\in L^2_{\mathbb{F}}(0,T;\mathbb{R}^m)} \sup_{\lambda\in L_{\mathcal{F}_{T}}^2(\Omega;\mathbb{R}^\ell)}L(u,\lambda),\quad\ \forall\ \lambda\in L_{\mathcal{F}_{T}}^2(\Omega;\mathbb{R}^\ell).
\end{equation}
Combining \eqref{eq 23} with \eqref{eq 3.11+}--\eqref{eq 3.12+}, we obtain that $(\bar u,\bar \lambda)$ is a saddle point of $L$.

This completes the proof of Theorem \ref{th strong dual}.
\end{proof}

By  Theorem \ref{th strong dual}, to solve the constrained stochastic LQ problem (CSLQ), we can first find the optimal solution $\bar \lambda$ to its dual problem \eqref{dual_for_SLQ}. Then, by \eqref{USLQ}, (CSLQ) can be transformed into an unconstrained stochastic LQ problem with optimal parameter $\bar \lambda$, and, the optimal solution to problem (CSLQ) can be found by the standard method of unconstrained stochastic LQ problem.

Consider the  Riccati equation
\begin{align}\label{eq Riccati}
\left\{
\begin{array}{l}
d P(t)=-\Big[P(t) A(t)+A(t)^{\top} P(t)+C(t)^{\top} P(t) C(t)+Q(t)+\Lambda(t) C(t)+C(t)^{\top} \Lambda(t)\\
\ns
\ds
\qquad\qquad-L(t)^{\top} K(t)^{-1} L(t)\Big] d t+\Lambda(t) d W(t), \quad t \in[0, T], \\
\ns
\ds
P(T)=G,
\end{array}
\right.
\end{align}
and the backward stochastic  differential equation
\begin{equation}\label{eq 261} 	
\left\{
\begin{array}{l}	
\mathrm{d}\varphi_{\lambda}(t)=-\Big[
\Big(A(t)^{\top} -L(t)^{\top}K(t)^{-1}B(t)^{\top}\Big)\varphi_{\lambda}(t) \\[2mm]
\qquad\qquad\quad+\Big(C(t)^{\top} -L(t)^{\top}K(t)^{-1}D(t)^{\top}\Big)\psi_{\lambda}(t)\Big]\mathrm{d}t+\psi_{\lambda}(t)\mathrm{d}W(t)
, \quad  t\in [0,T],\\[2mm]
\ns
\ds
\varphi_{\lambda}(T)=M^{\top}\lambda.
\end{array}
\right.
\end{equation}
Here,
\begin{align}\label{eq xishu}
L(t)\triangleq B(t)^{\top} P(t)+D(t)^{\top} P(t) C(t)+D(t)^{\top} \Lambda(t), \quad K(t)\triangleq R(t)+D(t)^{\top} P(t) D(t) .
\end{align}

By (A1)--(A3) and \cite[Theorem 6.1]{SunXiongYong2021}, the Riccati equation \eqref{eq Riccati} admits a unique solution  $\big(P(\cdot), \Lambda(\cdot)\big) \in L_{\mathbb{F}}^{\infty}\left(\Omega ; C\left([0, T], \mathbb{S}^{n}\right)\right) \times L_{\mathbb{F}}^{2}\left(0, T ; \mathbb{S}^{n}\right)$  such that  $K(t) \geq \delta I_{m} , a.e. \ t \in[0, T] , a.s.$ for some  $\delta>0 $. Similar to \cite{zhang}, when the solution $\big(P(\cdot), \Lambda(\cdot)\big) $ satisfies the  regularity condition
\begin{align}\label{eq zhnegzexing}
K(t)^{-1} L(t) \in L_{\mathbb{F}}^{\infty}\left(\Omega; L^{2}\left(0, T ; \mathbb{R}^{m \times n}\right)\right),
\end{align}
the dual functional $d(\cdot)$ has a much simpler expression.

\begin{proposition}\label{prop exp_lambda}
Suppose that  (A1)--(A4) hold. Let $\big(P(\cdot),\Lambda(\cdot)\big)$ be the solution to Riccati equation \eqref{eq Riccati} satisfying the regularity condition \eqref{eq zhnegzexing}. Then
\begin{equation}\label{lambda}
d(\lambda)
=\frac{1}{2}\langle P(0)x,x\rangle+\langle \varphi_{\lambda}(0),x\rangle-\mathbb{E}\langle b,\lambda\rangle-\frac{1}{2}\mathbb{E}\int_{0}^{T}\Big|K(t)^{-\frac{1}{2}}\big[B(t)^{\top}\varphi_{\lambda}(t)+D(t)^{\top}\psi_{\lambda}(t)\big]\Big|^2\mathrm{d}t,
\end{equation}
where $(\varphi_{\lambda}(\cdot), \psi_{\lambda}(\cdot))\in L_{\mathbb{F}}^{2}\left(\Omega; C\left([0, T]; \mathbb{R}^{n}\right)\right)\times L_{\mathbb{F}}^{2}\left(0, T; \mathbb{R}^{n}\right)$ is an adapted solution to equation \eqref{eq 261} and $L(\cdot)$ and $K(\cdot)$ are defined by \eqref{eq xishu}. In addition,
\begin{align}\label{bar_lambda}
\bar{u}_{\lambda}(t)\triangleq -K(t)^{-1}\Big[L(t)\bar{X}^{x,\bar{u}_{\lambda}}(t)+B(t)^{\top}\varphi_{\lambda}(t)+D(t)^{\top}\psi_{\lambda}(t)\Big]
\end{align}
is the  feedback optimal solution of the parameterized stochastic LQ problem, i.e.,
\begin{equation}\label{UCESLQ_lambda}
L(\bar{u}_{\lambda},\lambda)=\inf_{u\in L^{2}_{\mathbb{F}}(0,T;\mathbb{R}^m)}L(u,\lambda).
\end{equation}
\end{proposition}
\begin{proof}
By the solvability of Riccati equation \eqref{eq Riccati} and It\^{o}'s formula, we
obtain that
\begin{align}\label{eq 47}
&\mathbb{E}\langle P(T)X^{x,u}(T),X^{x,u}(T)\rangle-\langle P(0)x,x\rangle\nonumber\\
&=\mathbb{E}\int_{0}^{T}\Big[\langle (L(t)^{\top}K(t)^{-1}L(t)-Q(t)    )X^{x,u}(t),X^{x,u}(t)\rangle\nonumber\\
&\qquad+2\langle L(t)X^{x,u}(t),u(t)\rangle+\langle D(t)^{\top}P(t)D(t)u(t),u(t)\rangle\Big]dt.
\end{align}
Also, applying It\^{o}'s formula to $\langle \varphi_{\lambda}(\cdot),X^{x,u}(\cdot) \rangle $, we get
\begin{align}\label{eq 30}
&\mathbb{E}\langle \varphi_{\lambda}(T),X^{x,u}(T)\rangle-\langle \varphi_{\lambda}(0),x\rangle\nonumber\\
&=\mathbb{E}\int_{0}^{T}\Big[\langle L(t)^{\top}K(t)^{-1}B(t)^{\top}\varphi_{\lambda}(t)
+L(t)^{\top}K(t)^{-1}D(t)^{\top}\psi_{\lambda}(t), X^{x,u}(t)\rangle\nonumber\\
&\qquad+\langle B(t)^{\top}\varphi_{\lambda}(t)+ D(t)^{\top}\psi_{\lambda}(t),u(t)\rangle\Big]\mathrm{d}t.
\end{align}
Combining \eqref{eq 47} with \eqref{eq 30}, we have
\begin{align}\label{eq 31}
L(u,\lambda)
&=J(u)+\mathbb{E}\langle MX^{x,u}(T)-b,\lambda\rangle\nonumber\\[-0.2em]
&=\frac{1}{2}\mathbb{E}\int_{0}^{T} \langle Q(t)X^{x,u}(t),X^{x,u}(t)\rangle+\langle R(t)u(t),u(t)\rangle \mathrm{d}t\nonumber\\[+0.5em]
&\qquad+\frac{1}{2}\mathbb{E}\langle GX^{x,u}(T),X^{x,u}(T)\rangle+\mathbb{E}\langle MX^{x,u}(T)-b,\lambda\rangle\nonumber\\[-0.2em]
&=\frac{1}{2}\langle P(0)x,x\rangle+\langle \varphi_{\lambda}(0),x\rangle- \mathbb{E}\langle b,\lambda\rangle\nonumber\\[-0.2em]
&\qquad+\frac{1}{2}\mathbb{E}\int_{0}^{T}\Big[\langle L(t)^{\top}K(t)^{-1}L(t)X^{x,u}(t),X^{x,u}(t)\rangle+\langle K(t)u(t),u(t)\rangle\nonumber\\[-0.2em]
&\qquad+2\langle L(t)X^{x,u}(t)+B(t)^{\top}\varphi_{\lambda}(t)+D(t)^{\top}\psi_{\lambda}(t),u(t)\rangle\nonumber\\[-0.2em]
&\qquad+2\langle L(t)^{\top}K(t)^{-1}B(t)^{\top}\varphi_{\lambda}(t)+L(t)^{\top}K(t)^{-1}D(t)^{\top}\psi_{\lambda}(t),X^{x,u}(t)\rangle\Big]dt\nonumber\\
&=\frac{1}{2}\langle P(0)x,x\rangle+\langle \varphi_{\lambda}(0),x\rangle-\mathbb{E}\langle b,\lambda\rangle\nonumber\\[-0.2em]
&\qquad+\frac{1}{2}\mathbb{E}\int_{0}^{T} \Big|K(t)^{\frac{1}{2}}\big[u(t)+K(t)^{-1}(L(t)X^{x,u}(t)+B(t)^{\top}\varphi_{\lambda}(t)
+D(t)^{\top}\psi_{\lambda}(t))\big]\Big|^2\mathrm{d}t\nonumber\\[-0.2em]
&\qquad
-\frac{1}{2}\mathbb{E}\int_{0}^{T}
\Big|K(t)^{-\frac{1}{2}}\big[B(t)^{\top}\varphi_{\lambda}(t)+D(t)^{\top}\psi_{\lambda}(t)\big]\Big|^2dt.
\end{align}
Therefore,
\begin{equation*}
\bar{u}_{\lambda}(t)=-K(t)^{-1}\Big[L(t)\bar{X}^{x, 	\bar{u}_{\lambda}}(t)+B(t)^{\top}\varphi_{\lambda}(t)+D(t)^{\top}\psi_{\lambda}(t))\Big]
\end{equation*}
is the unique optimal solution to the unconstrained stochastic LQ problem \eqref{UCESLQ_lambda} and
\begin{align*}
d(\lambda)=&\inf_{u\in L^{2}_{\mathbb{F}}(0,T;\mathbb{R}^m)}L(u,\lambda)\\
=&\frac{1}{2}\langle P(0)x,x\rangle+\langle \varphi_{\lambda}(0),x\rangle-\mathbb{E}\langle b,\lambda\rangle-\frac{1}{2}\mathbb{E}\int_{0}^{T}\Big|K(t)^{-\frac{1}{2}}\big[B(t)^{\top}\varphi_{\lambda}(t)+D(t)^{\top}\psi_{\lambda}(t)\big]\Big|^2\mathrm{d}t.
\end{align*}
This completes the proof of Proposition \ref{prop exp_lambda}.
\end{proof}

\begin{theorem}\label{th bar lambda}
Suppose that  (A1)--(A4) hold. Let $\big(P(\cdot),\Lambda(\cdot)\big)$ be the solution to Riccati equation \eqref{eq Riccati} satisfying the regularity condition \eqref{eq zhnegzexing}. Then the optimal control of (CSLQ) is
\begin{equation*}
\bar{u}(t)=-K(t)^{-1}\Big[L(t)\bar{\mathbb{X}}^{x,\bar \lambda}(t)+B(t)^{\top}\varphi_{\bar\lambda}(t)+D(t)^{\top}\psi_{\bar\lambda}(t)\Big], \ a.s.,\vspace{-0.5em}
\end{equation*}
where the optimal parameter $\bar \lambda$ is the solution to the first-order necessary condition for the  dual problem \eqref{dual_for_SLQ} that
\begin{equation}\label{1st conditon for dual}
M\bar{\mathbb{X}}^{x,\bar \lambda}(T)-b=0,\quad a.s.
\end{equation}
and $\bar{\mathbb{X}}^{x,\bar \lambda}$ is the solution to the equation
\begin{equation}\label{eq Pi}
\left\{\!\!\!
\begin{array}{ll}
d\bar{\mathbb{X}}^{x,\bar \lambda}(t)\!=\!\Big[(A(t)-B(t)K(t)^{-1}L(t))\bar{\mathbb{X}}^{x,\bar \lambda}(t)
   -B(t)K(t)^{-1}(B(t)^{\top}\varphi_{\bar{\lambda}}(t)+D(t)^{\top}\psi_{\bar{\lambda}}(t))\Big] \mathrm{d}t\\[+0.5em]
\qquad\qquad  +\Big[(C(t)-D(t)K(t)^{-1}L(t))\bar{\mathbb{X}}^{x,\bar \lambda}(t) -D(t)K(t)^{-1}(B(t)^{\top}\varphi_{\bar{\lambda}}(t)\!+\!D(t)^{\top}\psi_{\bar{\lambda}}(t))\Big]\! \mathrm{d}W(t),\\ [+0.5em]
\qquad\qquad\qquad\qquad\qquad\qquad\qquad\qquad\qquad\qquad\qquad\qquad\qquad\qquad\qquad\qquad\qquad \qquad t\!\in\! [0,T],\\
\bar{\mathbb{X}}^{x,\bar \lambda}(0)=x.
\end{array}
\right.
\end{equation}
\end{theorem}
\begin{proof}
Let $\bar \lambda$ be an optimal solution to the  dual problem \eqref{dual_for_SLQ}. Then, by \eqref{lambda}  and the optimality of $\bar \lambda$, for any $ \mu\in L_{\mathcal{F}_{T}}^2(\Omega;\mathbb{R}^\ell)$, we obtain
\begin{align}\label{eq linear1}
0=&\lim\limits_{\varepsilon\to 0^{+}}\frac{d(\bar{\lambda}+\varepsilon\mu)-d(\bar{\lambda})}{\varepsilon}\nonumber\\ =&-\mathbb{E}\int_{0}^{T}\big\langle K(t)^{-1}(B(t)^{\top}\varphi_{\bar{\lambda}}(t)+D(t)^{\top}\psi_{\bar{\lambda}}(t)),B(t)^{\top}\varphi_{\mu}(t)+D(t)^{\top}\psi_{\mu}(t)\big\rangle\mathrm{d}t\nonumber\\
&\quad-\mathbb{E}\langle b,\mu\rangle+\langle \varphi_{\mu}(0),x\rangle,
\end{align}
where $(\varphi_{\mu},\psi_{\mu})$ is the solution to \eqref{eq 261} with final datum $M^{\top}\lambda$ replaced by $M^{\top}\mu$.

Applying It\^{o}'s formula to $\langle\bar{\mathbb{X}}^{x,\bar \lambda}(\cdot),\varphi_{\mu}(\cdot)\rangle$, we get
\begin{align}\label{eq 3.26}
&\mathbb{E}\langle M\bar{\mathbb{X}}^{x,\bar \lambda}(T),\mu\rangle\nonumber\\
&=\mathbb{E}\langle \bar{\mathbb{X}}^{x,\bar \lambda}(T),M^{\top}\mu\rangle\nonumber\\
&=\langle \varphi_{\mu}(0),x\rangle-\mathbb{E}\int_{0}^{T}\big\langle K(t)^{-1}(B(t)^{\top}\varphi_{\bar{\lambda}}(t)+D(t)^{\top}
\psi_{\bar{\lambda}}(t)),B(t)^{\top}\varphi_{\mu}(t)+D(t)^{\top}\psi_{\mu}(t)\big\rangle\mathrm{d}t.
\end{align}

Combining \eqref{eq linear1} with \eqref{eq 3.26}, we obtain that
\begin{align*}
M\bar{\mathbb{X}}^{x,\bar \lambda}(T)-b=0,\quad a.s.
\end{align*}
is the first-order necessary condition for the optimal solution $\bar \lambda$ to the dual problem \eqref{dual_for_SLQ}.
Then the conclusion follows from Proposition \ref{prop exp_lambda}.
\end{proof}

\begin{remark}
By Theorem  \ref{th bar lambda}, we obtain a closed-form solution to the constrained stochastic LQ problem (CSLQ). However, it is in general difficult to gain the optimal parameter $\bar \lambda$ by solving the first-order necessary condition (\ref{1st conditon for dual}).
\end{remark}

\section{Augmented Lagrangian method}\label{sec3.2}

In this section, we propose an augmented Lagrangian  method (ALM) for solving (CSLQ) and prove its convergence.

For any $\lambda\in L_{\mathcal{F}_{T}}^{2}\left(\Omega ; \mathbb{R}^{\ell}\right)$ and $u\in L^{2}_{\mathbb{F}}(0,T;\mathbb{R}^m)$, the augmented Lagrangian functional for (CSLQ) is defined by
\begin{equation}\label{eq 5}
\begin{split}
L_{\rho}(u,\lambda)&\triangleq J(u)+\mathbb{E}\langle \lambda,MX^{x,u}(T)-b\rangle+\frac{\rho}{2}\mathbb{E}\big| MX^{x,u}(T)-b\big|^2\\
&=\frac{1}{2}\mathbb{E}\Bigg[\int_{0}^{T}\langle Q(t)X^{x,u}(t),X^{x,u}(t)\rangle+\langle R(t)u(t),u(t)\rangle\mathrm{d}t+\langle GX^{x,u}(T),X^{x,u}(T)\rangle\Bigg]\\
		&\qquad+\mathbb{E}\langle \lambda,MX^{x,u}(T)-b\rangle+\frac{\rho}{2}\mathbb{E}\big| MX^{x,u}(T)-b\big|^2,
\end{split}
\end{equation}
where~$\rho>0$ is called the penalty parameter.

The ALM for (CSLQ) is defined as follows.
\begin{algorithm}
\caption{ALM for (CSLQ)}
\begin{algorithmic}
    \State Step 0.  Let $k=0$. Choose  $\lambda^0\in L_{\mathcal{F}_{T}}^2(\Omega;\mathbb{R}^\ell)$, $u^{0}\in L_{\mathbb{F}}^2(0,T;\mathbb{R}^m)$, $  \{r^{k}\}_{k=0}^{\infty}\subset(0,\infty)$.
 \State Step 1.  Calculate $u^{k+1}$ such that
      \begin{equation}\label{eq suanfa2}
 L_{\rho}(u^{k+1},\lambda^k)=\inf_{u\in L_{\mathbb{F}}^2(0,T;\mathbb{R}^m)}L_{\rho}(u,\lambda^k).
    \end{equation}
\State Step 2. Update the multiplier by
\begin{align}\label{eq suanfa3}
    \lambda^{k+1}=\lambda^k+r^{k}\big(MX^{x,u^{k+1}}(T)-b\big).
\end{align}
Let $k:=k+1$ and return to Step 1.
\end{algorithmic}
\end{algorithm}

\begin{remark}
The unconstrained stochastic LQ sub-problem
(\ref{eq suanfa2}) can be solved by constructing its optimal feedback solution. Let us
consider the   Riccati equation
\begin{equation} \label{Riccatiequ}	
\left\{\!\!\!
\begin{array}{l}
dP_{\rho}(t)=-\big[P_{\rho}(t)A(t)+A(t)^{\top}P_{\rho}(t)+C(t)^{\top}P_{\rho}(t)C(t)
+Q(t)+\Lambda_{\rho}(t)C(t)+C(t)^{\top}\Lambda_{\rho}(t)\\[+0.5em]
\qquad\qquad -L_{\rho}(t)^{\top}K_{\rho}(t)^{-1}L_{\rho}(t)\big]dt+\Lambda_{\rho}(t)dW(t),
\quad  t\in [0,T],\\[+0.5em]
P_{\rho}(T)=G+\rho M^{\top}M
\end{array}
\right.
\end{equation}
and the backward stochastic differential equation
\begin{equation}\label{eq 12} 	
		\left\{\!\!\!
		\begin{array}{l}
		d\varphi_{\rho,\lambda^k}(t)=-\Big[\big(A(t)^{\top}-L_{\rho}(t)^{\top}K_{\rho}(t)^{-1}B(t)^{\top}\big)\varphi_{\rho,\lambda^k}(t)\\[+0.5em]
\qquad\qquad\quad\ +\big(C(t)^{\top}-L_{\rho}(t)^{\top}K_{\rho}(t)^{-1}D(t)^{\top}\big)\psi_{\rho,\lambda^k}(t)\Big]dt
+\psi_{\rho,\lambda^k}(t)dW(t),\quad  t\in [0,T],\\[+0.5em]
			\varphi_{\rho,\lambda^k}(T)=M^{\top}\lambda^k-\rho M^{\top}b,
		\end{array}
		\right.
	\end{equation}
where
\begin{equation}\label{eq xishup}
L_{\rho}(t)\!=\!B(t)^{\top} P_{\rho}(t)\!+\!D(t)^{\top} P_{\rho}(t) C(t)\!+\!D(t)^{\top} \Lambda_{\rho}(t), \  K_{\rho}(t)\!=\!R(t)\!+\!D(t)^{\top} P_{\rho}(t) D(t),\ t\!\in\![0,T].
\end{equation}
Let us define the functional
\begin{align*}
J^{0}_{\rho}(u)\triangleq\frac{1}{2}\mathbb{E}\Bigg[\int_{0}^{T}\langle Q(t)X^{0,u}(t),X^{0,u}(t)\rangle+\langle R(t)u(t),u(t)\rangle\mathrm{d}t+\langle (G+\rho M^{\top}M)X^{0,u}(T),X^{0,u}(T)\rangle\Bigg],
\end{align*}
where $u\in L^{2}_{\mathbb{F}}(0,T;\mathbb{R}^m)$ and the state $X^{0,u}(\cdot)$ is the solution to control system (\ref{eq controlsys}) with control $u(\cdot)$ and initial datum $0$.
Under condition (A3),
\begin{align*}
J^{0}_{\rho}(u)\ge J^{0}(u)\ge \delta\mathbb{E} \int_{0}^{T}|u(t)|^{2} \mathrm{d} t, \qquad u\in L^{2}_{\mathbb{F}}(0,T;\mathbb{R}^m).
\end{align*}
Similar to Lemma \ref{lemma existence}, we can prove that  $L_{\rho}(\cdot,\lambda^k)$ is strongly convex on $L^{2}_{\mathbb{F}}(0,T;\mathbb{R}^m)$ and, the unconstrained stochastic LQ sub-problem (\ref{eq suanfa2}) admits a unique optimal solution. In addition, by \cite[Theorem 6.1]{SunXiongYong2021}, the perturbed Riccati equation (\ref{Riccatiequ}) has unique solution $\big(P_{\rho}(\cdot),\Lambda_{\rho}(\cdot)\big)\in L^{\infty}_{\mathbb{F}}\big(\Omega;C([0,T];\mathbb{S}^n)\big)\times L^{2}_{\mathbb{F}}(0,T;\mathbb{S}^n)$ such that
\begin{align*}
R(t)+D(t)^{\top}P_{\rho}(t)D(t)\ge\delta^{'}I_{m}, \quad \text{a.e. } t\in [0,T],\ a.s.
\end{align*}	
for some $\delta^{'}>0$. Then, the BSDE (\ref{eq 12}) also admits a unique solution $(\varphi_{\rho,\lambda^k}, \psi_{\rho,\lambda^k})$. By the standard theory of unconstrained stochastic LQ problem, it can be shown that the optimal solution to the unconstrained stochastic LQ sub-problem (\ref{eq suanfa2}), denoted by $u_{\rho,\lambda^k}$, has the feedback form
\begin{align*}
u_{\rho,\lambda^k}(t)=-K_{\rho}(t)^{-1}\Big[L_{\rho}(t)X^{x,u_{\rho,\lambda^k}}(t)+B(t)^{\top}\varphi_{\rho,\lambda^k}(t)
+D(t)^{\top}\psi_{\rho,\lambda^k}(t)\Big], \quad \text{a.e. } t\in [0,T], \text{a.s.}
\end{align*}
Therefore, to solve the unconstrained stochastic LQ sub-problem (\ref{eq suanfa2}), we only need to solve the Riccati equation (\ref{Riccatiequ}) and the BSDE  (\ref{eq 12}).
\end{remark}

Now, let us prove the convergence of the ALM for (CSLQ). To this end, we need some technical lemmas. First, we prove that the saddle points of the Lagrangian functional $L(\cdot,\cdot)$ coincide with those of the augmented Lagrangian functional $L_{\rho}(\cdot,\cdot)$.
\begin{lemma}\label{lemmas}
$(\bar{u},\bar{\lambda})$ is a saddle point of $L(\cdot,\cdot)$ if and only if it is a saddle point of $L_{\rho}(\cdot,\cdot)$.
\end{lemma}
\begin{proof}
If $(\bar{u},\bar{\lambda})$ is a saddle point of $L(\cdot,\cdot)$, then
\begin{equation}\label{eq 4.7}
L(\bar{u},\lambda)\le L(\bar{u},\bar{\lambda})\le  L(u,\bar{\lambda}),\quad \forall\   u\in L_{\mathbb{F}}^2(0,T;\mathbb{R}^m), \forall\  \lambda\in L_{\mathcal{F}_{T}}^2(\Omega;\mathbb{R}^\ell).
\end{equation}
Clearly, \eqref{eq 4.7} holds true only if
\begin{align*}
M\bar{X}^{x,\bar{u}}(T)-b=0,\ a.s.
\end{align*}
Then, we have
$$L_{\rho}(\bar{u},\lambda)=L(\bar{u},\lambda)+\frac{\rho}{2}\mathbb{E}\big| M\bar{X}^{x,\bar{u}}(T)-b\big|^2=L(\bar{u},\lambda) $$
and
\begin{align*}
L_{\rho}(\bar{u},\bar{\lambda})=L(\bar{u},\bar{\lambda}).
\end{align*}
It implies that
\begin{align}\label{1}
L_{\rho}(\bar{u},\lambda)\le L_{\rho}(\bar{u},\bar{\lambda}),\qquad \forall\ \lambda\in L_{\mathcal{F}_{T}}^2(\Omega;\mathbb{R}^\ell).
\end{align}
In addition, by
\begin{align*}
L(\bar{u},\bar{\lambda})\le L(u,\bar{\lambda}),\qquad \forall\ u\in L_{\mathbb{F}}^2(0,T;\mathbb{R}^m)
\end{align*}
and
\begin{align*}
\mathbb{E}\big| M\bar{X}^{x,\bar{u}}(T)-b\big|^2=0\le \mathbb{E}\big| MX^{x,u}(T)-b\big|^2,\quad \forall\ u\in L_{\mathbb{F}}^2(0,T;\mathbb{R}^m),
\end{align*}
we have
\begin{equation*}
L(\bar{u},\bar{\lambda})+\frac{\rho}{2}\mathbb{E}\big| M\bar{X}^{x,\bar{u}}(T)-b\big|^2\le L(u,\bar{\lambda})+ \frac{\rho}{2}\mathbb{E}\big| MX^{x,u}(T)-b\big|^2, \quad \forall\ u\in L_{\mathbb{F}}^2(0,T;\mathbb{R}^m).
\end{equation*}
Therefore,
\begin{align}\label{2}
L_{\rho}(\bar{u},\bar{\lambda})\le L_{\rho}(u,\bar{\lambda}), \quad \forall\ u\in L_{\mathbb{F}}^2(0,T;\mathbb{R}^m).
\end{align}
Combining \eqref{1} with \eqref{2}, we obtain that $(\bar{u},\bar{\lambda})$ is a saddle point of $L_{\rho}(\cdot,\cdot)$.

Next, suppose that $(\bar{u},\bar{\lambda})$ is a saddle point of $L_{\rho}(\cdot,\cdot)$, i.e.,
\begin{align*}
L_{\rho}(\bar{u},\lambda)\le L_{\rho}(\bar{u},\bar{\lambda})\le  L_{\rho}(u,\bar{\lambda}),\quad \forall\   u\in L_{\mathbb{F}}^2(0,T;\mathbb{R}^m), \forall\  \lambda\in L_{\mathcal{F}_{T}}^2(\Omega;\mathbb{R}^\ell).
\end{align*}
Then, we can also obtain that $M\bar{X}^{x,\bar{u}}(T)-b=0$, a.s. It implies that
\begin{align}\label{4}
L(\bar{u},\lambda)=L_{\rho}(\bar{u},\lambda)\le L_{\rho}(\bar{u},\bar{\lambda})=L(\bar{u},\bar{\lambda}),\qquad \forall \lambda\in L_{\mathcal{F}_{T}}^2(\Omega;\mathbb{R}^\ell),
\end{align}
and
\begin{align*}
J(\bar{u})=L(\bar{u},\bar{\lambda})=L_{\rho}(\bar{u},\bar{\lambda})\le L_{\rho}(u,\bar{\lambda}),\quad \forall u\in L_{\mathbb{F}}^2(0,T;\mathbb{R}^m).
\end{align*}

Then, by the convexity of $J(\cdot)$, for any $u\in L_{\mathbb{F}}^2(0,T;\mathbb{R}^m)$ and  $\theta\in(0,1)$, we have
\begin{align*}
J(\bar{u})&\le L_{\rho}(\bar{u}+\theta(u-\bar{u}),\bar{\lambda})\\
&=J\big(\bar{u}+\theta(u-\bar{u})\big)+\mathbb{E}\langle \bar{\lambda},MX^{x,\bar{u}+\theta(u-\bar{u})}(T)-b\rangle+\frac{\rho}{2}\mathbb{E}\big| MX^{x,\bar{u}+\theta(u-\bar{u})}(T)-b\big|^2\\
&\le \theta J(u)+(1-\theta)J(\bar{u})+\mathbb{E}\langle \bar{\lambda},\theta MX^{x,u}(T)+(1-\theta)M\bar{X}^{x,\bar{u}}(T)-b\rangle\\
&\qquad+\frac{\rho}{2}\mathbb{E}\big|\theta MX^{x,u}(T)+(1-\theta)M\bar{X}^{x,\bar{u}}(T)-b\big|^2\\
&=\theta J(u)+(1-\theta)J(\bar{u})+\theta\mathbb{E}\langle \bar{\lambda}, MX^{x,u}(T)-b\rangle+\frac{\rho\theta^2}{2}\mathbb{E}\big|  MX^{x,u}(T)-b\big|^2.
\end{align*}
It implies that
\begin{align*}
0\le J(u)-J(\bar{u})+\mathbb{E}\langle \bar{\lambda}, MX^{x,u}(T)-b\rangle+\frac{\rho\theta}{2}\mathbb{E}\big|MX^{x,u}(T)-b\big|^2,\quad \forall u\in L_{\mathbb{F}}^2(0,T;\mathbb{R}^m).
\end{align*}
Letting $\theta\to 0^{+}$, we have
\begin{align*}
J(\bar{u})\le J(u)+\mathbb{E}\langle \bar{\lambda}, MX^{x,u}(T)-b\rangle, \quad \forall u\in L_{\mathbb{F}}^2(0,T;\mathbb{R}^m),
\end{align*}
i.e.,
\begin{align}\label{5}
L(\bar{u},\bar{\lambda})\le L(u,\bar{\lambda}),\qquad \forall u\in L_{\mathbb{F}}^2(0,T;\mathbb{R}^m).
\end{align}
Combining \eqref{4} with \eqref{5}, we obtain that $(\bar{u},\bar{\lambda})$ is a saddle point of $L(\cdot,\cdot)$.

This completes the proof of Lemma \ref{lemmas}.
\end{proof}

\begin{lemma}\label{lemmass}
Let (A1)--(A4) hold. Then, $\bar{u}$ is an optimal control of (CSLQ) if and only if there is $\bar{\lambda}$ such that $(\bar{u},\bar{\lambda})$ is a saddle point of $L_{\rho}(\cdot,\cdot)$.
\end{lemma}
\begin{proof}
It has been proved in Theorem \ref{th strong dual} that if $\bar{u}$ is an optimal control of (CSLQ), then there is $\bar{\lambda}$ such that $(\bar{u},\bar{\lambda})$ is a saddle point of $L(\cdot,\cdot)$. Thus, by Lemma \ref{lemmas}, $(\bar{u},\bar{\lambda})$ is a saddle point of $L_{\rho}(\cdot,\cdot)$.

On the other hand, if $(\bar{u},\bar{\lambda})$ is a saddle point of $L_{\rho}(\cdot,\cdot)$, then
\begin{align*}
M\bar{X}^{x,\bar{u}}(T)-b=0,\ a.s.
\end{align*}
and
\begin{equation*}
J(\bar{u})=L_{\rho}(\bar u,\bar{\lambda}) \le\inf_{u\in L_{\mathbb{F}}^2(0,T;\mathbb{R}^m)}L_{\rho}(u,\bar{\lambda})\le \inf_{\substack{ u\in L^{2}_{\mathbb{F}}(0,T;\mathbb{R}^m)\\ MX^{x,u}(T)-b=0}}L_{\rho}(u,\bar{\lambda})= \inf_{\substack{ u\in L^{2}_{\mathbb{F}}(0,T;\mathbb{R}^m) \\ MX^{x,u}(T)-b=0}}J(u).
\end{equation*}
Therefore, $\bar u$ is an optimal control of (CSLQ).

This completes the proof of Lemma \ref{lemmass}.
\end{proof}

By Lemma \ref{lemmass}, the constrained stochastic LQ problem (CSLQ) is equivalent to the saddle point problem of  $L_{\rho}(\cdot,\cdot)$.

For the unconstrained stochastic LQ sub-problem \eqref{eq suanfa2}, the following first-order  necessary and sufficient condition  holds true.
\begin{lemma}\label{3}
Let (A1)--(A3) hold. Then, $u_{\rho,\lambda^{k}}$ is the unique optimal control to the unconstrained stochastic LQ sub-problem (\ref{eq suanfa2}) if and only if
$$
R(t)u_{\rho,\lambda^{k}}(t)-B(t)^{\top}p_{\rho,\lambda^{k}}(t)-D(t)^{\top}q_{\rho,\lambda^{k}}(t)=0,\qquad a.e.\ t\in[0,T],\ a.s.
$$
where $(p_{\rho,\lambda^{k}},q_{\rho,\lambda^{k}}\big)$ is the solution to the BSDE
\begin{equation}\label{eq 3.38} 	
\left\{\!\!\!
\begin{array}{l} dp_{\rho,\lambda^{k}}(t)\!=\!-\Big[\!A(t)^{\top}\!p_{\rho,\lambda^{k}}(t)\!+\!C(t)^{\top}\!q_{\rho,\lambda^{k}}(t)
\!-\!Q(t)X^{x,u_{\rho,\lambda^{k}}}(t)\!\Big]dt\!+\!q_{\rho,\lambda^{k}}(t)dW(t),\ t\!\in\! [0,T],\\[+1em]
p_{\rho,\lambda^{k}}(T)=-(G+\rho M^{\top}M)X^{x,u_{\rho,\lambda^{k}}}(T)-M^{\top}\lambda^{k}+\rho M^{\top}b,
\end{array}
\right.
\end{equation}
where $X^{x,u_{\rho,\lambda^{k}}}$ is the solution to the linear control system (\ref{eq controlsys}) with control $u_{\rho,\lambda^{k}}$ and initial datum $x$.
\end{lemma}
\begin{proof}
For any $\varepsilon>0$ and $v\in L^{2}_{\mathbb{F}}(0,T;\mathbb{R}^m)$,
\begin{align*}
&L_{\rho}(u_{\rho,\lambda^{k}}+\varepsilon v,\lambda^{k})-L_{\rho}(u_{\rho,\lambda^{k}},\lambda^{k})\\
&=\varepsilon\bigg[\mathbb{E}\int_{0}^{T}\big\langle Q(t)X^{x,u_{\rho,\lambda^{k}}}(t),X^{0,v}(t)\big\rangle+\big\langle R(t)u_{\rho,\lambda^{k}}(t),v(t)\big\rangle dt\\
&\qquad+\mathbb{E}\big\langle (G+\rho M^{\top}M)X^{x,u_{\rho,\lambda^{k}}}(T)+M^{\top}\lambda^{k}-\rho M^{\top}b,X^{0,v}(T)\big\rangle\bigg]\\
&\qquad+\frac{\varepsilon^2}{2}\Bigg[\mathbb{E}\int_{0}^{T}\big\langle Q(t)X^{0,v}(t),X^{0,v}(t)\big\rangle+\big\langle R(t)v(t),v(t)\big\rangle dt\\
&\qquad+\mathbb{E}\big\langle (G+\rho M^{\top}M)X^{0,v}(T),X^{0,v}(T)\big\rangle\bigg].
\end{align*}
Here $X^{0,v}$ is the solution to the linear control system \eqref{eq controlsys} with control $v$ and initial datum $0$.
If $u_{\rho,\lambda^{k}}$ is an optimal control, then
\begin{align}\label{eq bianfen}
0\le&\lim\limits_{\varepsilon\to 0^{+}}\frac{L_{\rho}(u_{\rho,\lambda^{k}}+\varepsilon v,\lambda^{k})-L_{\rho}(u_{\rho,\lambda^{k}},\lambda^{k})}{\varepsilon}\nonumber\\
=&\mathbb{E}\int_{0}^{T}\big\langle Q(t)X^{x,u_{\rho,\lambda^{k}}}(t),X^{0,v}(t)\big\rangle+\big\langle R(t)u_{\rho,\lambda^{k}}(t),v(t)\big\rangle dt\nonumber\\
\qquad&+\mathbb{E}\big\langle (G+\rho M^{\top}M)X^{x,u_{\rho,\lambda^{k}}}(T)+M^{\top}\lambda^{k}-\rho M^{\top}b,X^{0,v}(T)\big\rangle,\quad \forall v\in L^{2}_{\mathbb{F}}(0,T;\mathbb{R}^m).
\end{align}
By It\^{o}'s formula, we have
\begin{align}\label{eq formula}
&\mathbb{E}\big\langle (G+\rho M^{\top}M)X^{x,u_{\rho,\lambda^{k}}}(T)+M^{\top}\lambda^{k}-\rho M^{\top}b,X^{0,v}(T)\big\rangle
\nonumber\\
&=-\mathbb{E}\big\langle p_{\rho,\lambda^{k}}(T),X^{0,v}(T)\big\rangle\nonumber\\
&=-\mathbb{E}\int_{0}^{T}\big\langle B(t)^{\top} p_{\rho,\lambda^{k}}(t)+D(t)^{\top} q_{\rho,\lambda^{k}}(t),v(t)\big\rangle\mathrm{d}t-\mathbb{E}\int_{0}^{T}\big\langle Q(t)X^{x,u_{\rho,\lambda^{k}}}(t),X^{0,v}(t)\big\rangle dt.
\end{align}
Combining \eqref{eq bianfen} with \eqref{eq formula}, we obtain
\begin{align*}
\mathbb{E}\int_{0}^{T}\big\langle R(t)u_{\rho,\lambda^{k}}(t)-B(t)^{\top}p_{\rho,\lambda^{k}}(t)-D(t)^{\top}q_{\rho,\lambda^{k}}(t),v(t)\big\rangle dt\ge 0,\quad \forall\ v\in L^{2}_{\mathbb{F}}(0,T;\mathbb{R}^m),
\end{align*}
which implies
\begin{align*}
R(t)u_{\rho,\lambda^{k}}(t)-B(t)^{\top}p_{\rho,\lambda^{k}}(t)-D(t)^{\top}q_{\rho,\lambda^{k}}(t)=0,\qquad a.e.\ t\in[0,T], a.s.
\end{align*}
This proves the necessity.

Next, assume that $u_{\rho,\lambda^{k}}$ satisfies the condition
\begin{align*}
R(t)u_{\rho,\lambda^{k}}(t)-B(t)^{\top}p_{\rho,\lambda^{k}}(t)-D(t)^{\top}q_{\rho,\lambda^{k}}(t)=0,\qquad a.e.\ t\in[0,T], a.s.
\end{align*}
Then
\begin{align}\label{hengdeng}
\mathbb{E}\int_{0}^{T}\big\langle R(t)u_{\rho,\lambda^{k}}(t)-B(t)^{\top}p_{\rho,\lambda^{k}}(t)-D(t)^{\top}q_{\rho,\lambda^{k}}(t),v(t)\big\rangle dt= 0, \quad \forall\ v\in L^{2}_{\mathbb{F}}(0,T;\mathbb{R}^m).
\end{align}
By \eqref{eq formula} and \eqref{hengdeng}, we obtain
\begin{align}\label{eng}
&\mathbb{E}\int_{0}^{T}\big\langle Q(t)X^{x,u_{\rho,\lambda^{k}}}(t),X^{0,v}(t)\big\rangle+\big\langle R(t)u_{\rho,\lambda^{k}}(t),v(t)\big\rangle dt\nonumber\\
&+\mathbb{E}\big\langle (G+\rho M^{\top}M)X^{x,u_{\rho,\lambda^{k}}}(T)+M^{\top}\lambda^{k}-\rho M^{\top}b,X^{0,v}(T)\big\rangle= 0.
\end{align}
Then, by condition (A3), for any $v\in L^{2}_{\mathbb{F}}(0,T;\mathbb{R}^m)$
\begin{align*}
L_{\rho}(u_{\rho,\lambda^{k}}+v,\lambda^{k})&=L_{\rho}(u_{\rho,\lambda^{k}},\lambda^{k})+\mathbb{E}\int_{0}^{T} \big\langle Q(t)X^{0,v}(t),X^{0,v}(t)\big\rangle+\big\langle R(t)v(t),v(t)\big\rangle dt\\
&\quad+\mathbb{E}\big\langle (G+\rho M^{\top}M)X^{0,v}(T),X^{0,v}(T)\big\rangle\\
&\ge L_{\rho}(u_{\rho,\lambda^{k}},\lambda^{k})+J^{0}(v)\\
&\ge L_{\rho}(u_{\rho,\lambda^{k}},\lambda^{k})+\delta\mathbb{E}\int_{0}^{T}|v(t)|^2\mathrm{d}t,
\end{align*}
which implies that $u_{\rho,\lambda^{k}}$ is the unique optimal control to the unconstrained stochastic LQ sub-problem \eqref{eq suanfa2}. This proves the sufficiency.
\end{proof}

We have the following first-order necessary and sufficient condition for the constrained stochastic LQ problem (CSLQ).

\begin{lemma}\label{lemmasss}
Let (A1)--(A4) hold. Then, $\bar{u}$ is an optimal control of (CSLQ) if and only if there is $\bar \lambda\in L_{\mathcal{F}_{T}}^2(\Omega;\mathbb{R}^\ell)$ such that
\begin{equation}\label{1st conditon for primal}	
\left\{
\begin{array}{l} M\bar{X}^{x,\bar{u}}(T)-b=0,\ a.s.\\[+0.5em]			R(t)\bar{u}(t)-B(t)^{\top}p_{\rho,\bar{\lambda}}(t)
-D(t)^{\top}q_{\rho,\bar{\lambda}}(t)=0,\quad a.e.\ t\in [0,T], a.s.
\end{array}
\right.
\end{equation}
where $\big(p_{\rho,\bar{\lambda}}(\cdot),q_{\rho,\bar{\lambda}}(\cdot)\big)$ is the solution to BSDE (\ref{eq 3.38}) with $\lambda$ replaced by $\bar{\lambda}$.
\end{lemma}
\begin{proof}
By Lemma \ref{lemmass}, $\bar{u}$ is an optimal control of (CSLQ) if and only if there is $\bar{\lambda}$ such that $(\bar{u},\bar{\lambda})$ is a saddle point of $L_{\rho}(\cdot,\cdot)$. Clearly, $(\bar{u},\bar{\lambda})$ is a saddle point of $L_{\rho}(\cdot,\cdot)$ if and only if
\begin{align*}
M\bar{X}^{x,\bar{u}}(T)-b=0,\qquad a.s.,
\end{align*}
and
\begin{align*}
L_{\rho}(\bar{u},\bar{\lambda})=\inf_{u\in L^{2}_{\mathbb{F}}(0,T;\mathbb{R}^m)}L_{\rho}(u,\bar{\lambda}).
\end{align*}
Then, the conclusion follows by a similar argumentation in Lemma \ref{3}.
\end{proof}

We are now in a position to establish the main result of this section, namely the convergence of ALM for (CSLQ).

\begin{theorem}\label{th 4.1}
Suppose that (A1)--(A4) hold true and let $r^{0}>0$. Then, for any $\{r^{k}\}$ such that $0<r^{0}\le r^{k}\le 2\rho$ and any $\lambda^{0}\in L_{\mathcal{F}_{T}}^2(\Omega;\mathbb{R}^{\ell})$, the sequence $\{u^k\}$ generated by the ALM  converges strongly to the unique solution $\bar{u}$ of (CSLQ) in $L^{2}_{\mathbb{F}}(0,T;\mathbb{R}^m)$.
\end{theorem}
\begin{proof}
By Lemma \ref{3}, for any $k,\ u^{k+1}$ satisfies the first-order necessary condition
\begin{align*}
R(t)u^{k+1}(t)-B(t)^{\top}p_{\rho,\lambda^{k}}(t)-D(t)^{\top}q_{\rho,\lambda^{k}}(t)=0,\quad a.e.\ t\in [0,T], a.s.,
\end{align*}
where $\big(p_{\rho,\lambda^{k}}(\cdot),q_{\rho,\lambda^{k}}(\cdot)\big)$ is the solution to BSDE \eqref{eq 3.38}.  Then, by It\^{o}'s formula, for any $v\in L^{2}_{\mathbb{F}}(0,T;\mathbb{R}^m)$, we have
\begin{align*}
&\mathbb{E}\int_{0}^{T}\big\langle Q(t)X^{x,u^{k+1}}(t),X^{0,v}(t)\big\rangle+\big\langle R(t)u^{k+1}(t),v(t)\big\rangle dt\nonumber\\
&+\mathbb{E}\big\langle (G+\rho M^{\top}M)X^{x,u^{k+1}}(T),X^{0,v}(T)\big\rangle+\mathbb{E}\big\langle M^{\top}\lambda^k-\rho M^{\top}b,X^{0,v}(T)\big\rangle= 0.
\end{align*}
Especially, for $v=\bar{u}-u^{k+1}$, we have
\begin{align}\label{6}
&\mathbb{E}\int_{0}^{T}\big\langle Q(t)X^{x,u^{k+1}}(t),X^{0,\bar{u}-u^{k+1}}(t)\big\rangle+\big\langle R(t)u^{k+1}(t),\bar{u}(t)-u^{k+1}(t)\big\rangle dt\nonumber\\
&+\mathbb{E}\big\langle (G+\rho M^{\top}M)X^{x,u^{k+1}}(T),X^{0,\bar{u}-u^{k+1}}(T)\big\rangle+\mathbb{E}\big\langle M^{\top}\lambda^k-\rho M^{\top}b,X^{0,\bar{u}-u^{k+1}}(T)\big\rangle= 0.
\end{align}
Similarly, by Lemma \ref{lemmasss} and It\^{o}'s formula, we have
\begin{align}\label{7}
&\mathbb{E}\int_{0}^{T}\big\langle Q(t)\bar{X}^{x,\bar{u}}(t),X^{0,u^{k+1}-\bar{u}}(t)\big\rangle+\big\langle R(t)\bar{u}(t),u^{k+1}(t)-\bar{u}(t)\big\rangle dt\nonumber\\
&+\mathbb{E}\big\langle (G+\rho M^{\top}M)\bar{X}^{x,\bar{u}}(T),X^{0,u^{k+1}-\bar{u}}(T)\big\rangle+\mathbb{E}\big\langle M^{\top}\bar{\lambda}-\rho M^{\top}b,X^{0,u^{k+1}-\bar{u}}(T)\big\rangle= 0.
\end{align}
By \eqref{6}--\eqref{7}  and the linearity of control system \eqref{eq controlsys}, we obtain that
\begin{align*}
&\mathbb{E}\int_{0}^{T} \big\langle Q(t)X^{0,u^{k+1}-\bar{u}}(t),X^{0,u^{k+1}-\bar{u}}(t)\big\rangle+\big\langle R(t)(u^{k+1}(t)-\bar{u}(t)),u^{k+1}(t)-\bar{u}(t)\big\rangle dt\\
&+\mathbb{E}\big\langle GX^{0,u^{k+1}-\bar{u}}(T), X^{0,u^{k+1}-\bar{u}}(T)\big\rangle\\
&+\mathbb{E}\big\langle M^{\top}(\lambda^k-\bar{\lambda}),X^{0,u^{k+1}-\bar{u}}(T)\big\rangle+\rho\mathbb{E}\big\langle M^{\top}MX^{0,u^{k+1}-\bar{u}}(T),X^{0,u^{k+1}-\bar{u}}(T)\big\rangle= 0.
\end{align*}
Letting $v^{k+1}=u^{k+1}-\bar{u}$, we have
\begin{align}\label{eq 4.20+}
&\mathbb{E}\big\langle \lambda^k-\bar{\lambda},MX^{0,v^{k+1}}(T)\big\rangle\nonumber\\
&=-\mathbb{E}\int_{0}^{T} \big\langle Q(t)X^{0,v^{k+1}}(t),X^{0,v^{k+1}}(t)\big\rangle+\big\langle R(t)v^{k+1}(t),v^{k+1}(t)\big\rangle dt\nonumber\\
&\quad -\mathbb{E}\big\langle GX^{0,v^{k+1}}(T), X^{0,v^{k+1}}(T)\big\rangle-\rho\mathbb{E}\big\langle M^{\top}MX^{0,v^{k+1}}(T),X^{0,v^{k+1}}(T)\big\rangle.
\end{align}
By  \eqref{eq suanfa3}, \eqref{eq 4.20+}, condition (A3) and the fact that $M\bar{X}^{x,\bar{u}}(T)-b=0$  a.s., we have
\begin{align}\label{eq 4.20}
\mathbb{E}\big| \lambda^{k+1}-\bar{\lambda}\big|^2
&=\mathbb{E}\big| \lambda^{k}+r^k(MX^{x,u^{k+1}}(T)-b)-\bar{\lambda}\big|^2\nonumber\\
&=\mathbb{E}\big| \lambda^{k}-\bar{\lambda}\big|^2+2r^k\mathbb{E}\big\langle\lambda^{k}-\bar{\lambda}, MX^{0,v^{k+1}}(T)\big\rangle+(r^k)^2\mathbb{E}\big| MX^{0,v^{k+1}}(T)\big|^2\nonumber\\
&=\mathbb{E}\big| \lambda^{k}-\bar{\lambda}\big|^2-2r^kJ^{0}(v^{k+1})-2r^k\rho\mathbb{E}\big| MX^{0,v^{k+1}}(T)\big|^2+(r^k)^2\mathbb{E}\big| MX^{0,v^{k+1}}(T)\big|^2\nonumber\\
&\le\mathbb{E}\big| \lambda^{k}-\bar{\lambda}\big|^2-2r^k\delta\mathbb{E}\int_{0}^{T}|v^{k+1}(t)|^2\mathrm{d}t-r^k(2\rho-r^k)\mathbb{E}\big| MX^{0,v^{k+1}}(T)\big|^2.
\end{align}
This proves that the sequence $\{\|\lambda^{k+1}-\bar{\lambda}\|_{L_{\mathcal{F}_{T}}^2(\Omega;\mathbb{R}^\ell)}\}$ is decreasing and bounded below by $0$, hence it is convergent.
In addition, by \eqref{eq 4.20}, we have
\begin{align*}
0\le2r^0\delta\mathbb{E}\int_{0}^{T}|v^{k+1}(t)|^2\mathrm{d}t\le\mathbb{E}\big| \lambda^{k}-\bar{\lambda}\big|^2-\mathbb{E}\big| \lambda^{k+1}-\bar{\lambda}\big|^2.
\end{align*}
Letting $k\to +\infty$, we have
$$\| u^{k+1}-\bar{u}\|_{L^{2}_{\mathbb{F}}(0,T;\mathbb{R}^m)}^2=\| v^{k+1}\|_{L^{2}_{\mathbb{F}}(0,T;\mathbb{R}^m)}^2\to 0.$$

This completes the proof of Theorem \ref{th 4.1}.
\end{proof}

\section{The characterization of condition (A4)}\label{Sec 4}
In this section, we shall give a sufficient and necessary condition and some sufficient conditions for  condition (A4). Some basic ideas are from the fundamental controllability argumentation of \cite{Peng1994,LuZhang2021}.

In order to characterize the condition (A4), let us consider the following norm optimal control problem:
\begin{equation*}
\left\{
\begin{array}{ll}
\min &\mathbb{E}\displaystyle\int_{0}^{T}|u(t)|^2\mathrm{d}t, \\[+0.4em]
\ns
\ds		
\text{s.t. }&{u\in L^2_{\mathbb{F}}(0,T;\mathbb{R}^m)},\\[+0.4em]
&MX^{x,u}(T)-b=0, \ a.s.,\tag{NP}
\end{array}
\right.
\end{equation*}
where $X^{x,u}(\cdot)$ is a solution to the control system $\eqref{eq controlsys}$ with control $u$ and initial datum $x$.
Clearly, the problem (NP) is a special case of (CSLQ) with $Q(\cdot)\equiv0$, $G\equiv 0$ and $R(\cdot)\equiv I_{m}>0$. Furthermore, $(P(t),\Lambda(t))\equiv 0$ is the solution to its Riccati equation
\begin{equation*}
\left\{\!\!\!
\begin{array}{l}
dP(t)=-\big[P(t)A(t)+A(t)^{\top}P(t)+C(t)^{\top}P(t)C(t)+\Lambda(t)C(t)+C(t)^{\top}\Lambda(t)\\
\ns
\ds
\qquad\qquad -L(t)^{\top}K(t)^{-1}L(t)\big]dt+\Lambda(t)dW(t),
\quad  t\in [0,T],\\[+0.3em]
P(T)=0.
\end{array}
\right.
\end{equation*}
Here $L(\cdot)$ and $K(\cdot)$ are defined by \eqref{eq xishu}.

Define the Lagrangian functional of (NP) by
\begin{align*}
L(u,\lambda)\triangleq\frac{1}{2}\mathbb{E}\int_{0}^{T}|u(t)|^2\mathrm{d}t+\mathbb{E}\langle\lambda, MX^{x,u}(T)-b\rangle,\quad \forall\ u\in L^2_{\mathbb{F}}(0,T;\mathbb{R}^m),  \forall\ \lambda\in L^{2}_{\mathcal{F}_{T}}(\Omega;\mathbb{R}^\ell) .
\end{align*}
For any $\lambda\in L^{2}_{\mathcal{F}_{T}}(\Omega;\mathbb{R}^\ell)$, the equation \eqref{eq 261} reduces to
\begin{equation}\label{eq 34} 	
\left\{
\begin{array}{l}
d\varphi_{\lambda}(t)=-\Big\{
A(t)^{\top} \varphi_{\lambda}(t) +C(t)^{\top} \psi_{\lambda}(t)\Big\}dt
+\psi_{\lambda}(t) dW(t), \quad  t\in [0,T],\\[+0.5em]
\ns
\ds
\varphi_{\lambda}(T)=M^{\top}\lambda.
\end{array}
\right.
\end{equation}

By \eqref{eq 34} and It\^{o}'s formula, we have
\begin{align*}
L(u,\lambda)
&=\frac{1}{2}\mathbb{E}\int_{0}^{T}|u(t)|^2\mathrm{d}t+\mathbb{E}\langle \lambda, MX^{x,u}(T)-b\rangle \\
&=\frac{1}{2}\mathbb{E}\int_{0}^{T}|u(t)|^2\mathrm{d}t+\langle \varphi_{\lambda}(0), x\rangle-\mathbb{E}\langle\lambda,b\rangle+\mathbb{E}\int_{0}^{T}\langle B(t)^{\top}\varphi_{\lambda}(t)+D(t)^{\top}\psi_{\lambda}(t), u(t)\rangle dt\\
&=\frac{1}{2}\mathbb{E}\int_{0}^{T}|u+ B(t)^{\top}\varphi_{\lambda}(t)+D(t)^{\top}\psi_{\lambda}(t)|^2\mathrm{d}t+\langle \varphi_{\lambda}(0), x\rangle\\
&\quad-\mathbb{E}\langle\lambda,b\rangle-\frac{1}{2}\mathbb{E}\int_{0}^{T}| B(t)^{\top}\varphi_{\lambda}(t)+D(t)^{\top}\psi_{\lambda}(t)|^2 dt.
\end{align*}
Then, the Lagrangian dual functional for (NP) is
\begin{equation}\label{eq 351}
\begin{split}
d(\lambda)
&\triangleq \inf_{u\in L^{2}_{\mathbb{F}}(0,T;\mathbb{R}^m)}L(u,\lambda)=-\frac{1}{2}\mathbb{E}\int_{0}^{T}| B(t)^{\top}\varphi_{\lambda}(t)+D(t)^{\top}\psi_{\lambda}(t)|^2 dt+\langle \varphi_{\lambda}(0), x\rangle-\mathbb{E}\langle\lambda,b\rangle.
\end{split}
\end{equation}
and the minimal solution to the optimization problem in \eqref{eq 351} is
$$u_{\lambda}(t)=- B(t)^{\top}\varphi_{\lambda}(t)-D(t)^{\top}\psi_{\lambda}(t),\qquad a.e.\ t\in[0,T], \ a.s.$$
Define the dual problem of (NP) as follows:
\begin{equation*}
\left\{
\begin{array}{ll}
\max & d(\lambda)=-\dfrac{1}{2}\mathbb{E}\displaystyle\int_{0}^{T}|B(t)^{\top}\varphi_{\lambda}(t)+D(t)^{\top}\psi_{\lambda}(t)|^2\mathrm{d}t+\langle \varphi_{\lambda}(0), x\rangle-\mathbb{E}\langle\lambda,b\rangle,\\[2mm]
\ns
\ds
\text{s.t. } &\lambda\in L_{\mathcal{F}_{T}}^2(\Omega;\mathbb{R}^\ell).\tag{ND}
\end{array}
\right.
\end{equation*}

We have the following result.
\begin{proposition}\label{prop 5.1}
If $\bar{\lambda}$ is the optimal solution to (ND), then
\begin{equation}\label{bar u for NP} \bar{u}(t)=-B(t)^{\top}\varphi_{\bar{\lambda}}(t)-D(t)^{\top}\psi_{\bar \lambda}(t) \qquad a.e.\ t\in[0,T], \ a.s.
 \end{equation}
is the optimal solution of (NP), where $(\varphi_{\bar{\lambda}},\psi_{\bar \lambda})$ is the solution to (\ref{eq 34}).
\end{proposition}

\begin{proof}
Let $\bar \lambda, \mu\in L_{\mathcal{F}_{T}}^2(\Omega;\mathbb{R}^\ell)$, $(\varphi_{\bar{\lambda}},\psi_{\bar \lambda})$ and $(\varphi_{\mu},\psi_{\mu})$ be the solutions to $\eqref{eq 34}$ with final datum $M^{\top}\bar\lambda$ and $M^{\top}\mu$, respectively. If $\bar{\lambda}$ is an optimal solution to (ND), then, for any $\mu\in L_{\mathcal{F}_{T}}^2(\Omega;\mathbb{R}^\ell)$,
\begin{align}\label{eq 36}
0&=\langle\nabla d(\bar \lambda),\mu\rangle\nonumber\\
&=\langle\varphi_{\mu}(0),x\rangle-\mathbb{E}\langle b,\mu\rangle-\int_{0}^{T}\langle B(t)^{\top}\varphi_{\bar{\lambda}}(t)+D(t)^{\top}\psi_{\bar \lambda}(t), B(t)^{\top}\varphi_{\mu}(t)+D(t)^{\top}\psi_{  \mu}(t)\rangle\mathrm{d}t.
\end{align}
Let $\bar{X}^{x,\bar{u}}$ be the solution to the controlled system \eqref{eq controlsys} with   control $\bar{u}$ defined by \eqref{bar u for NP} and initial datum $x$. By \eqref{bar u for NP},\eqref{eq 36} and It\^{o}'s formula, for any $\mu\in L_{\mathcal{F}_{T}}^2(\Omega;\mathbb{R}^\ell)$,
\begin{equation}\label{eq 37}
\begin{split}
&\mathbb{E}\langle M\bar{X}^{x,\bar{u}}(T)-b,\mu\rangle\\
&=\mathbb{E} \langle M^{\top}\mu, \bar{X}^{x,\bar{u}}(T)\rangle-\mathbb{E}\langle b,\mu\rangle\\
&=\langle \varphi_{\mu}(0),x\rangle+\mathbb{E}\int_{0}^{T}\langle\bar{u}(t),  B(t)^{\top}\varphi_{\mu}(t)+D(t)^{\top}\psi_{  \mu}(t)\rangle\mathrm{d}t-\mathbb{E}\langle b,\mu\rangle\\
&=0.
\end{split}
\end{equation}
Due to the arbitrariness of $\mu$, we obtain
$M\bar{X}^{x,\bar{u}}(T)-b=0,\ a.s.$
This proves that the $\bar{u}$ defined by \eqref{bar u for NP} is a feasible control.

Next, we prove the optimality of $\bar u$. Replacing $\mu$ by $\bar\lambda$ in \eqref{eq 37}, we obtain that
\begin{equation}\label{eq 45}
0=\langle \varphi_{\bar \lambda}(0),x\rangle-\mathbb{E}\int_{0}^{T}|\bar{u}(t)|^2\mathrm{d}t-\mathbb{E}\langle b,\bar\lambda\rangle.
\end{equation}
For any $u\in L^{2}_{\mathbb{F}}(0,T;\mathbb{R}^m)$ with the corresponding state $X^{x,u}$ such that $MX^{x,u}(T)-b=0,\ a.s.$,  by It\^{o}'s formula,
\begin{equation}\label{eq 44}
\begin{split}
&\mathbb{E}\langle X^{x,u}(T), M^{\top}\bar\lambda\rangle=\langle \varphi_{\bar \lambda}(0),x\rangle+\mathbb{E}\int_{0}^{T}\langle u(t),  B(t)^{\top}\varphi_{\bar \lambda}(t)+D(t)^{\top}\psi_{  \bar \lambda}(t)\rangle\mathrm{d}t.
\end{split}
\end{equation}
Combining \eqref{eq 44} with \eqref{eq 45}, we obtain that
\begin{equation*}
\begin{split}
\mathbb{E}\int_{0}^{T}|\bar{u}(t)|^2\mathrm{d}t&=\mathbb{E}\langle MX^{x,u}(T)-b, \bar\lambda\rangle -\mathbb{E}\int_{0}^{T}\langle u(t),  B(t)^{\top}\varphi_{\bar \lambda}(t)+D(t)^{\top}\psi_{  \bar \lambda}(t)\rangle\mathrm{d}t\\
&=-\mathbb{E}\int_{0}^{T}\langle u(t),  B(t)^{\top}\varphi_{\bar \lambda}(t)+D(t)^{\top}\psi_{  \bar \lambda}(t)\rangle\mathrm{d}t\\
&\le \Big[\mathbb{E}\int_{0}^{T}| u(t)|^2\mathrm{d}t\Big]^{\frac{1}{2}} \Big[\mathbb{E}\int_{0}^{T}| \bar u(t)|^2\mathrm{d}t\Big]^{\frac{1}{2}}.
\end{split}
\end{equation*}
This proves the optimality of $\bar u$.
\end{proof}
The following theorem gives a necessary and sufficient condition for $u\mapsto MX^{x,u}(T)$ to be a surjection.
\begin{theorem}\label{th observess estm}
Suppose that (A1) holds true. Then, $u\mapsto MX^{x,u}(T)$ is a surjection if and only if there is $c>0$ such that
\begin{equation}\label{eq 381}
\mathbb{E}\int_{0}^{T}| B(t)^{\top}\varphi_{\lambda}(t)+D(t)^{\top}\psi_{\lambda}(t)|^2\mathrm{d}t\geq c\mathbb{E}|\lambda|^2, \quad \forall\ \lambda\in L_{\mathscr{F}_{T}}^2(\Omega;\mathbb{R}^\ell),
\end{equation}
where $(\varphi_{\lambda}(\cdot),\psi_{\lambda}(\cdot))$ is an adapted solution to (\ref{eq 34}).
\end{theorem}
\begin{proof}
Let us fix arbitrarily $\alpha\in L_{\mathcal{F}_{T}}^2(\Omega;\mathbb{R}^\ell)$ and define
\begin{equation*}
\hat{d}_{\alpha}(\lambda)\triangleq -\frac{1}{2}\mathbb{E}\int_{0}^{T}|B(t)^{\top}\varphi_{\lambda}(t)+D(t)^{\top}\psi_{\lambda}(t)|^2\mathrm{d}t+\langle \varphi_{\lambda}(0), x\rangle-\mathbb{E}\langle\lambda,\alpha\rangle.
\end{equation*}
If inequality \eqref{eq 381} holds, then $\hat{d}_{\alpha}(\lambda)$ is coercive. Meanwhile, $\hat{d}_{\alpha}(\lambda)$ is a continuous concave functional. Thus (ND) has an optimal solution  $\bar{\lambda}_{\alpha}$.
Similar to Proposition \ref{prop 5.1}, we conclude that
$$\bar u_{\alpha}(t)=-B(t)^{\top}\varphi_{\bar{\lambda}_{\alpha}}(t)+D(t)^{\top}\psi_{\bar{\lambda}_{\alpha}}(t), \qquad a.e.\ t\in[0,T], \ a.s.$$
is a minimal norm control satisfying
$$M\bar{X}^{x,\bar u_{\alpha}}(T)=\alpha,\quad a.s.$$
This proves the sufficiency.

Next, let us prove the necessity.
Suppose by contradiction that $u\mapsto MX^{x,u}(T)$ is surjective, but \eqref{eq 381} does not hold true. Then, there is $\{\lambda_{n}\}\subset L_{\mathcal{F}_{T}}^2(\Omega;\mathbb{R}^\ell)$ such that
\begin{align*}
\mathbb{E}\int_{0}^{T}\big|B(t)^{\top}\varphi_{\lambda_{n}}(t)+D(t)^{\top}\psi_{\lambda_{n}}(t)\big|^2\mathrm{d}t<\frac{\mathbb{E}|\lambda_{n}|^2}{n^2}.
\end{align*}
Set $\widehat{\lambda}_{n}=\frac{\sqrt{n}\lambda_{n}}{\big[\mathbb{E}|\lambda_{n}|^2\big]^{\frac{1}{2}}}$. Then $\mathbb{E}|\widehat{\lambda}_{n}|^2\to \infty$  and
\begin{align*}
\mathbb{E}\int_{0}^{T}\big|B(t)^{\top}\varphi_{\widehat{\lambda}_{n}}(t)+D(t)^{\top}\psi_{\widehat{\lambda}_{n}}(t)\big|^2\mathrm{d}t<\frac{n}{\mathbb{E}|\lambda_{n}|^2}\cdot\frac{\mathbb{E}|\lambda_{n}|^2}{n^2}=\frac{1}{n}\to 0.
\end{align*}
Since $u\mapsto MX^{x,u}(T)$ is surjective, for any $\alpha\in L_{\mathcal{F}_{T}}^2(\Omega;\mathbb{R}^\ell)$, there exists $u\in L^2_{\mathbb{F}}(0,T;\mathbb{R}^m)$ such that
\begin{equation*}
0=MX^{x,u}(T)-\alpha=MX^{0,u}(T)+MX^{x,0}(T)-\alpha.
\end{equation*}
Similarly to \eqref{eq 44}, we have
\begin{align*}
\mathbb{E}\langle \widehat{\lambda}_{n},\alpha-MX^{x,0}(T)\rangle&=\mathbb{E}\langle \widehat{\lambda}_{n},MX^{0,u}(T)\rangle\\[2mm]
&=\mathbb{E}\langle M^{\top}\widehat{\lambda}_{n},X^{0,u}(T)\rangle\\[2mm]
&=\mathbb{E}\int_{0}^{T}\langle u(t),B(t)^{\top}\varphi_{\widehat{\lambda}_{n}}(t)+D(t)^{\top}\psi_{\widehat{\lambda}_{n}}(t)\rangle\mathrm{d}t\\[2mm]
&\to 0\quad (n\to \infty).
\end{align*}
By the arbitrariness of $\alpha$, we obtain that $\widehat{\lambda}_{n}$ converges weakly to $0$. This implies that $\{\widehat{\lambda}_{n}\}$ is bounded, which contradicts to $\mathbb{E}|\widehat{\lambda}_{n}|^2\to \infty$ as $n\to \infty$. This proves \eqref{eq 381}.
\end{proof}

In the rest of this section, let us  discuss the special case that $A, B, C, D$ and $M$ are deterministic matrices.

\begin{lemma}\label{t}
Suppose that condition (A1) holds true. If $A, B, C, D$ and $M$ are deterministic matrices, then  the mapping $u\mapsto MX^{x,u}(T)$ is surjective only if $m\ge\ell$ and $\operatorname{Rank}(MD)= \ell$.
\end{lemma}
\begin{proof}
The proof is similar to that of \cite[Proposition 6.3]{LuZhang2021}, so we omit it.
\end{proof}

By Lemma \ref{t}, there are $K_{1}\in\mathbb{R}^{m\times m}, K_{2}\in\mathbb{R}^{m\times n}$ such that
\begin{align*}
    MDK_{1}=
\begin{pmatrix}
I_{\ell},0
\end{pmatrix},\qquad
MDK_{2}=-MC.
\end{align*}
Fix arbitrarily $z\in L^2_{\mathbb{F}}(0,T;\mathbb{R}^n)$, $v\in L^2_{\mathbb{F}}(0,T;\mathbb{R}^{m-\ell})$ and define
\begin{align}\label{control}
    u=K_{1}
\begin{pmatrix}
 Mz\\
v
\end{pmatrix}
+K_{2}X^{x,z,v}.
\end{align}
Substituting \eqref{control}   into the control system \eqref{eq controlsys}, we have
\begin{align*}
d MX^{x,z,v}(t) &=\Bigg[MAX^{x,z,v}(t)+MB\bigg(K_{1}
\begin{pmatrix}
 Mz(t)\\
v(t)
\end{pmatrix}
+K_{2}X^{x,z,v}(t)\bigg)\Bigg]dt\\
&\qquad+\Bigg[MCX^{x,z,v}(t)+MD\bigg(K_{1}
\begin{pmatrix}
 Mz(t)\\
v(t)
\end{pmatrix}
+K_{2}X^{x,z,v}(t)\bigg)\Bigg]dW(t)\\
&=\Bigg[M(A+BK_{2})X^{x,z,v}(t)+MBK_{1}\begin{pmatrix}
 Mz(t)\\
v(t)
\end{pmatrix}\Bigg]\mathrm{d}t+Mz(t) dW(t),\quad t\in[0,T].
\end{align*}
Setting $A_{1}=A+BK_{2}$ and letting $A_{2},\ B_{1}$ be the matrices such that
\begin{align*}
BK_{1}
\begin{pmatrix}
 Mz(t)\\
v(t)
\end{pmatrix}
=A_{2}z(t)+B_{1}v(t),
\end{align*}
we have
\begin{equation}\label{eq controlsy}
\left\{
\begin{array}{l}	\mathrm{d} MX^{x,z,v}(t)=\big(MA_{1}X^{x,z,v}(t)+MA_{2}z(t)+MB_{1}v(t)\big)\mathrm{d}t+Mz(t)\mathrm{d}W(t), \quad t\in[0,T]\\
\ns
\ds
MX(0)=Mx.
\end{array}
\right.
\end{equation}
From \eqref{eq controlsy} we obtain that $u\mapsto MX^{x,u}(T)$ is surjective if $(z,v)\mapsto MX^{x,z,v}(T)$ is surjective.

Consider the backward stochastic control system
\begin{equation}\label{eq backward}
\left\{
\begin{array}{l}	\mathrm{d} Y(t)=\big(A_{1}Y(t)+A_{2}z(t)+B_{1}v(t)\big)\mathrm{d}t+z(t)\mathrm{d}W(t), \quad t\in[0,T]\\
\ns
\ds
Y(T)=\eta_{T}.
\end{array}
\right.
\end{equation}
Clearly, $(z,v)\mapsto MX^{x,z,v}(T)$ is surjective if $\operatorname{Rank}(M)=\ell$ and \eqref{eq backward} is exactly controllable in the sense that for any $\eta_{T}\in L_{\mathcal{F}_{T}}^2(\Omega;\mathbb{R}^n)$  and $x\in \mathbb{R}^n$, there is $v\in L^2_{\mathbb{F}}(0,T;\mathbb{R}^{m-\ell})$ such that $Y(0;\eta_{T},v)=x$. Note that $\operatorname{Rank}(MD)=\ell$ and $m\ge\ell$ only if $\operatorname{Rank}(M)=\ell$. We have the following result.

\begin{theorem}
Suppose that (A1)  holds true. If $A,B,C,D$ and $M$ are deterministic matrices, then, $u\mapsto MX^{x,u}(T)$ is surjective if
\begin{enumerate} [ (i)]
\item $\operatorname{Rank}(MD)=\ell$;
\item (\ref{eq backward}) is exactly controllable.
\end{enumerate}
\end{theorem}

Furthermore, by \cite[Theorem 6.10]{LuZhang2021}, \eqref{eq backward} is exactly controllable if and only if
\begin{align*}
    \operatorname{Rank}\big([B_{1}, A_{1}B_{1}, A_{2}B_{1}, A_{1}^2 B_{1}, A_{1}A_{2}B_{1},A_{2}^2 B_{1}, A_{2}A_{1}B_{1}, \ldots]\big)=n.
\end{align*}
Then, we obtain the following rank condition for the subjectivity of $u\mapsto MX^{x,u}(T)$.

\begin{theorem}
Suppose that (A1) holds true. If $A,B,C,D$ and $M$ are deterministic matrices, then, $u\mapsto MX^{x,u}(T)$ is surjective if
\begin{enumerate} [ (i)]
\item $\operatorname{Rank}(MD)=\ell$;
\item $\operatorname{Rank}\big([B_{1}, A_{1}B_{1}, A_{2}B_{1}, A_{1}^2 B_{1}, A_{1}A_{2}B_{1},A_{2}^2 B_{1}, A_{2}A_{1}B_{1}, \ldots]\big)=n$.
\end{enumerate}
\end{theorem}

\end{document}